\documentclass[10pt,british]{article}
\usepackage{amsfonts}
\usepackage{latexsym}
\usepackage{amsmath}
\usepackage{amssymb}
\usepackage{amssymb}
\usepackage{amsthm}
\usepackage{mathrsfs}
\usepackage[pagebackref,colorlinks=true]{hyperref}

\newtheorem{thrm}{Theorem}[section]

\newtheorem{lemma}[thrm]{Lemma}

\newtheorem{prop}[thrm]{Proposition}

\newtheorem{cor}[thrm]{Corollary}

\newtheorem{rmrk}[thrm]{Remark}

\newtheorem{conv}[thrm]{Convention}

\newcommand{\newsection}{    % Numeration of eqs. is automatic
\setcounter{equation}{0}\section}
\def\appendix#1{\addtocounter{section}{1}\setcounter{equation}{0}
\renewcommand{\thesection}{\Alph{section}}
\section*{Appendix \thesection\protect\indent \parbox[t]{11.15cm}{#1}}
\addcontentsline{toc}{section}{Appendix \thesection\ \ \ #1}}

\newcommand{\be}{\begin{eqnarray}}
\newcommand{\ee}{\end{eqnarray}}
\newcommand{\bea}{\begin{eqnarray}}
\newcommand{\eea}{\end{eqnarray}}
\newcommand{\ba}{\begin{array}}
\newcommand{\ea}{\end{array}}

\newenvironment{frcseries}{\fontfamily{frc}\selectfont}{}

\def\d{\delta}

\font\mybb=msbm10 at 11pt

\def\bb#1{\hbox{\mybb#1}}

\def\bC {\bb{C}}

\def\sb {{\nabla}}
\def\LC{{\nabla^g}}
\def\ps{{\Psi^+}}
\def\sp{{\Psi^-}}

\def\ph{{\Phi}}
\def\ps{{\Psi^+}}
\def\sp{{\Psi^-}}

\begin{document}
%\begin{titlepage}
\begin{center}
%\today
\vspace*{-1.0cm}
%\hfill hep-th/yymmnnn \\
%\hfill UB-ECM-PF-06-43 \\

%\vspace{2.0cm} {\Large \bf Vanishing theorems  for Hermitian
%manifolds
% satisfying $\omega^\ell\wedge \partial\bar\partial \omega^k=0$
%} \\[.2cm]

%\vspace{1.0cm}
{\Large \bf Almost Calabi-Yau with torsion 6-manifolds \\[5pt] and the instanton condition
% satisfying $\omega^\ell\wedge \partial\bar\partial \omega^k=0$
} %\\[.2cm]

\vspace{0.5  cm}
 {\large Stefan  Ivanov${}^1$ and  Luis Ugarte$^2$}

\vspace{0.5cm}

${}^1$ University of Sofia, Faculty of Mathematics and
Informatics,\\ blvd. James Bourchier 5, 1164, Sofia, Bulgaria
\\and  Institute of Mathematics and Informatics,
Bulgarian Academy of Sciences\\
email: ivanovsp@fmi.uni-sofia.bg

\vspace{0.5cm}
${}^2$ Departamento de Matem\'aticas\,-\,I.U.M.A.\\
Universidad de Zaragoza\\
Campus Plaza San Francisco\\
50009 Zaragoza, Spain
email: ugarte@unizar.es

%\vspace{0.5cm}

\end{center}

\vskip 0.5 cm
\begin{abstract}
\noindent 
It is observed that on a compact  almost complex Calabi-Yau with torsion (ACYT) 6-manifold with co-closed Lee form the curvature of the torsion connection is an $SU(3)$-instanton if and only if the torsion is parallel with respect to the torsion connection. The same conclusion holds for any 
(non necessarily compact)
balanced ACYT 6-manifold. In particular, on a CYT 6-manifold the Strominger-Bismut connection is 
an $SU(3)$-instanton if and only if the torsion is parallel with respect to the Strominger-Bismut connection provided either the CYT 6-manifold  is compact with co-closed Lee form or it is a balanced CYT 6-manifold.

\medskip

Keywords: torsion connection, $SU(3)$ holonomy, almost Calabi-Yau with torsion $SU(3)$-instanton

\medskip

AMS MSC2020: 53C55, 53C21, 53C29, 53Z05
\end{abstract}

\vskip 0.5cm
\noindent{\bf Acknowledgements:} \vskip 0.1cm
\noindent We thank Jeff Streets for pointing out to us the connection of the $SU(3)$-Hull instanton on a CYT space with the Chern instanton of a generalized Hermitian metric on an associated vector bundle. 
The research of S.I.  is partially supported by Contract KP-06-H72-1/05.12.2023 with the National Science Fund of Bulgaria,  by Contract 80-10-61 / 27.5.2025  with the Sofia University "St.Kl.Ohridski".  
L.U. is partially supported by grant PID2023-148446NB-I00, funded by MICIU/AEI/10.13039/501100011033. 

\vskip 0.5cm

%Statements and Declarations: not applicable

%\end{titlepage}

\tableofcontents

\setcounter{section}{0}
\setcounter{subsection}{0}

%%%%%%%%%%%%%%%%%%%%%%%%%%%%%%%%%%%%%%%%%%%%%%%%%%%%%%%%%%%%%%%%%%%%%%%%%%

%%%%%%%%%%%%%%%%%%%%%%%%%%%%%%%%%%%%%%%%%%%%%%%%%%%%%%%%%%%%%%%%%%%%%%%%%%

\newsection{Introduction}
Riemannian manifolds with metric connections having totally skew-symmetric torsion and special holonomy received a lot of interest in mathematics and theoretical physics mainly from supersymmetric string theories and supergravity.  The main reason comes from the Hull-Strominger system which describes the supersymmetric background in heterotic string theories \cite{Str,Hull}. 

The number of preserved supersymmetries depends on the number of parallel spinors with respect to a metric connection $\sb$ with totally skew-symmetric torsion $T$. %  which is related to the Levi-Civita
%connection $\LC$ by $$ \nabla = \LC + \frac{1}{2}H. $$  the three form
The torsion 3-form $T$ is identified with the 3-form field strength in these theories. 
The presence of a $\nabla$-parallel spinor leads to restriction of the
holonomy group $Hol(\nabla)$ of the torsion connection $\nabla$.
Namely, $Hol(\nabla)$ has to be contained in $SU(n)$, $dim=2n$
\cite{Str,GMW,IP1,IP2,Car,BB,BBE}.

Complex non-K\"ahler geometries appear
in string compactifications  and are studied intensively for a long time
\cite{Str,GMW,GKMW,GMPW,GPap,BB,BBE}. 
Hermitian manifolds have widespread applications in both physics and differential geometry in connection with solutions to the Hull-Strominger system see \cite{LY,yau,yau1,FIUV,XS,PPZ,CPYau,CPY1,GRST,PPZ4,Ph} and references therein.
On a Hermitian manifold,  there exists   a  unique
connection which preserves  the Hermitian structure and has
totally skew-symmetric torsion tensor. Its existence and explicit expression first appeared in Strominger's seminal paper \cite{Str} in 1986 in connection with the heterotic supersymmetric string background, where he called it the H-connection. Three years later, Bismut formally discussed
and used this connection in his local index theorem paper \cite{bismut}, which leads to the name Bismut
connection in literature.
We call this connection the Strominger-Bismut connection. Note that the connection also appeared implicitly earlier (\cite{Yano}) and in some literature it was also called the KT connection (K\"ahler with torsion) or characteristic connection.  When the holonomy of the Strominger-Bismut connection is contained in $SU(n)$ one has the notion of Calabi-Yau manifold with torsion (CYT) spaces and these manifolds are of great interest in string theories and in mathematics.% since CYT appear in the Hull-Strominger system. %An interesting geometric explanation of the curvature of the Strominger-Bismut connection appears in the framework of generalized Ricci flow
%developed by Garcia-Fernandez and Streets \cite{GFS}. 
%Vanishing theorems for the Dolbeault cohomology on compact Hermitian non-K\"ahler manifold were found in terms of the Strominger-Bismut connection in  \cite{AI,IP2,IP1}.

 Some types of non-complex 6-manifolds have also been invented  in the
string theory due to the mirror symmetry and T-duality
\cite{GLMW,GM,Car,Car1,KLMS,KML,HLM}. 
Almost Hermitian manifolds with totally skew-symmetric Nijenhuis
tensor arise as target spaces of a class of (2,0)-supersymmetric
two-dimensional sigma models \cite{Pap}. For the consistency of
the theory, the Nijenhuis tensor has to be parallel with respect
to the torsion connection with holonomy contained in $SU(n)$. The known models are group manifolds as well as the nearly K\"ahler spaces. 
In general, an almost Hermitian manifold does not admit a metric connection preserving the almost Hermitian structure and having  totally skew-symmetric torsion. It is shown in \cite[Theorem~10.1]{FI} that this is equivalent to the condition that the Nijenhuis tensor of type (0,3) 
 is totally skew-symmetric, i.e. this is  the class $G_1$ in Gray-Hervella classification \cite{GrH}. In this case,  the connection is unique. This class contains the Hermitian manifolds with Strominger-Bismut connection as well as the nearly K\"ahler spaces. In the nearly K\"ahler case the torsion connection $\sb$ is the characteristic connection considered by Gray \cite{gray}, and the torsion $T$ and the Nijenhuis tensor are $\sb$-parallel (see e.g. \cite{Kir,BM}). If the holonomy of the torsion connection is contained in $SU(n)$, that is,  the Ricci 2-form $\rho$ of $\sb$  representing the first Chern class is identically zero,  then we have the notion of \emph{Almost Calabi-Yau manifold with torsion (briefly ACYT)}.

 As it was pointed out in \cite{Str}, the existence of a parallel spinor with respect to a metric connection $\sb$ with torsion 3-form, in dimension six,  leads to the restriction that its holonomy group has to be contained in $SU(3)$. This means one has to consider an $SU(3)$-structure, i.e. an almost Hermitian manifold $(M^6,g,J)$ with topologically trivial canonical bundle trivialized by a (3,0) with respect to $J$ form  $\Psi=\Psi^++\sqrt{-1}\,\Psi^-$\ endowed with a metric connection preserving the $SU(3)$ structure with totally skew-symmetric torsion. 
 
From the point of view of physics, the Hull-Strominger system is consisted of the Killing spinor equations and the anomaly cancellation condition.  The Killing spinor equations lead to consider compact  CYT space with an exact Lee form, i.e. the trace of the torsion should be an exact 1-form \cite{Str}. The anomaly cancellation condition expresses the exterior derivative of the torsion 3-form in terms of the difference of the first Pontrjagin 4-form of an instanton connection on a vector bundle and a connection on the tangent bundle. It is shown in \cite{I1} that solutions to the Hull-Strominger system solve the heterotic equations of motion in dimensions 5,6,7 and 8 if and only if the connection on the tangent bundle in the anomally cancellation condition is an instanton (see \cite{MS} and \cite{XS} for extension of this result to all dimensions). The instanton condition (Yang-Mills connection) is highly non-linear PDE and the physically relevant connection is the Hull connection which is defined as the metric connection with torsion $-T$. 
It is known that if there exists an SU(3)-instanton on a holomorphic vector bundle on a complex 6-manifold then  it is unique because of the non-K\"ahler version of the Donaldson-Uhlenbeck-Yau theorem  established by Li-Yau in \cite{LiY}.

The  main purpose of the paper is to investigate the $SU(3)$--instanton condition  of the Strominger-Bismut and $SU(3)$-Hull connections  on an ACYT 6-manifold.
%%%{\color{red} In the case of compact CYT 6-manifold, the %description  of the spaces with instanton Strominger-Bismut or $SU(3)$-Hull connections  are also compact CYT spaces with stable tangent bundle of degree zero due to the Kobayashi-Hitchin correspondence and the non-K\"ahler version of the Donaldson-Uhlenbeck-Yau theorem  established by Li-Yau in \cite{LiY}.??? Probably we can delete this red?}

\begin{thrm}\label{main} Let $(M,g,J,\Psi)$ be an ACYT 6-manifold  with  constant norm of the Nijenhuis tensor and $\sb$--parallel Lee form. 

Then the torsion connection $\sb$ is an $SU(3)$--instanton if and only if  the torsion is $\sb$-parallel, $\sb T=0$.
\end{thrm}
 It is known  from \cite[Proposition~4.2]{IS1} that if the ACYT space is balanced, i.e. the Lee form vanishes, then the Nijenhuis tensor  has constant norm and is $\sb$-parallel. 
As a consequences of Theorem~\ref{main}, we derive
\begin{cor}\label{main2} Let $(M,g,J,\Psi)$ be a balanced ACYT 6-manifold.

Then the torsion connection $\sb$ is an $SU(3)$--instanton if and only if  the torsion is $\sb$-parallel, $\sb T=0$.%  and $\sb$ be an instanton connection.  Then the torsion is $\sb$-parallel, $\sb T=0$.
%Let $(M,g,J,\Psi)$ be a compact balanced ACYT 6-manifold  and $\sb$ be an instanton connection.  Then the torsion is $\sb$-parallel, $\sb T=0$.
\end{cor}

In the compact case we show
\begin{thrm}\label{mainc}
Let $(M,g,J,\Psi)$ be a compact  ACYT 6-manifold with co-closed  Lee form.%, $\delta\theta=0$. 

Then the torsion connection $\sb$ is an $SU(3)$--instanton if and only if  the torsion is $\sb$-parallel, $\sb T=0$.
\end{thrm}
For a complex space, the next corollary is straightforward consequence of Theorem~\ref{main} and Theorem~\ref{mainc}.
\begin{cor}\label{comm}
 a) On a    CYT 6-manifold  with $\sb$-parallel Lee form the curvature of the Strominger-Bismut connection is an  $SU(3)$--instanton if and only if  the torsion is $\sb$-parallel, $\sb T=0$.

b)  On a   compact CYT 6-manifold  with co-closed Lee form the curvature of the Strominger-Bismut connection is an  $SU(3)$--instanton if and only if  the torsion is $\sb$-parallel, $\sb T=0$.

c) On a balanced CYT 6-manifold  the curvature of the Strominger-Bismut connection is an  $SU(3)$--instanton if and only if  the torsion is $\sb$-parallel, $\sb T=0$.
\end{cor}
%{\color{red} It is known that any compact complex manifold with holomorphically trivial canonical bundle admits a balanced CYT structure. Hence, in dimension 6, the CYT space with holomorphically trivial canonical has instanton Strominger-Bismut connection exactly when the torsion is parallel with respect to that connection}.

%In the complex case in dimension 6, the recent result in \cite{ABS} states that on a compact CYT space the torsion of the Strominger-Bismut connection is closed if and only if the connection is  flat and therefore it is a compact Lee group with the flat Cartan connection , biinvariant metric and invariant complex structure. As a corollary, we have
%\begin{cor}
%On a  CYT 6-manifold the $SU(3)$--Hull connection is an instanton if and only if the Strominger-Bismut connection is flat.
%\end{cor}

%}
Note that Hermitian and almost Hermitian manifolds of type $G_1$ with $\sb$-parallel torsion 3-form  are investigated in  \cite{AFS,Scho}, recently in \cite{CMS,ZZ2,ZZ3,ZZ4}  and a large number of examples are given there. 
On the Hermitian manifold, %$N=0$, 
the condition $\sb T=0$ is precisely  expressed in terms of the curvature of the Strominger-Bismut connection $\sb$ and its traces ( \cite[Theorem~1.1]{ZZ2,ZZ3}) and all compact CYT spaces with $\sb$-parallel torsion  are classified explicitly in \cite{ZZ4}.  The list presented  in \cite{ZZ2} coincides with the instanton CYT spaces in dimension 6 among the unimodular Lie algebras described in \cite{OU} 
 (see Section~\ref{Ex-Complex-case} for details).

For the $SU(3)$--Hull connection we have
\begin{thrm}\label{hul}
The $SU(3)$--Hull connection on an ACYT 6-manifold  is an $SU(3)$--instanton if and only if the torsion is closed, $dT=0$.
\end{thrm}
Theorem \ref{hul} connected  compact  6-dimensional ACYT having $SU(3)$--instanton Hull connection with type II string theories and the generalized Ricci flow developed by Garcia-Fernandez and Streets  (see \cite{GFS} and references therein) since these spaces  are steady generalized gradient Ricci solitons due to the results in \cite{IS1}. In \cite[Proposition~3.21]{GJS} (see also \cite{JS}), it is shown that for any pluriclosed ($dT=0$) Hermitian manifold, the curvature of its Hull connection is related to  the curvature of the Chern connection of the induced generalized Hermitian metric on the  associated orthogonal holomorphic vector bundle. This result, together with our Theorem~\ref{hul}, provides an interpretation of Hull instanton on a CYT 6-manifold as Chern instanton on the associated orthogonal holomorphic vector bundle. 
Moreover, in the 6-dimensional complex case, a structure theorem of  compact pluriclosed  CYT spaces  is recently presented in \cite[Theorem~1.1, Corollary~1.2]{ABS}, which in view of Theorem~\ref{hul} gives a description of the structure of compact CYT 6-dimensional manifolds with $SU(3)$--instanton Hull connection.

 It is known that on a compact  non-K\"ahler CYT space  with an exact  non-zero  Lee form the curvature of the  Chern connection is never an instanton. For completeness, a proof of this fact is included  in  Appendix~\ref{cins} based on \cite{IP2}.
 
\begin{conv}
Everywhere in the paper, we will make no difference between tensors and the corresponding forms via the metric as well as we will use Einstein summation conventions, i.e. repeated Latin indices are summed over from $1$ to $2n$ while repeated Greek indices are summed over from $1$ to $n$.
\end{conv}

\section{Preliminaries}
In this section, we recall some known curvature properties of a metric connection with totally skew-symmetric torsion on Riemannian manifolds as well as
the notions and existence of a metric linear connection preserving a given Almost Hermitian structure and having skew-symmetric torsion from \cite{I,FI,IS,IS1}.
\subsection{Metric connection with totally skew-symmetric  torsion and its curvature}%satisfying the  Riemannian first Bianchi identity}
Let $(M,g,\sb)$ be a Riemannian manifold with a metric connection $\sb$ and $T(X,Y)=\sb_XY-\sb_YX-[X,Y]$ be its torsion. The torsion of type (0,3) is denoted by the same letter and it is defined by $T(X,Y,Z)=g(T(X,Y),Z)$. The torsion is totally skew-symmetric if $T(X,Y,Z)=-T(X,Z,Y)$.

On a Riemannian manifold $(M,g)$ of dimension $n$ any metric connection $\sb$ with totally skew-symmetric torsion $T$ is connected with the Levi-Civita connection $\sb^g$ of the metric $g$ by
\begin{equation}\label{tsym}
\sb^g=\sb- \frac12T \quad leading\quad to \quad \LC T=\sb T+\frac12\sigma^T,
\end{equation}
where the 4-form $\sigma^T$, introduced in \cite{FI}, is defined by
 \begin{equation}\label{sigma}
% \begin{split}
 \sigma ^T(X,Y,Z,V)=\frac12\sum_{j=1}^n(e_j\lrcorner T)\wedge(e_j\lrcorner T)(X,Y,Z,V), %\\=\Big[H(X,Y,e_j)H(Z,V,e_j)+H(Y,Z,e_j)H(X,V,e_j)+H(Z,X,e_j)H(Y,V,e_j)\Big], 
 %\end{split}
\end{equation} 
   $(e_a\lrcorner T)(X,Y)=T(e_a,X,Y)$ is the interior multiplication and $\{e_1,\dots,e_n\}$ is an orthonormal  basis.

The properties of the 4-form $\sigma^T$ are studied in detail in \cite{AFF} where it is shown that $\sigma^T$ measures the `degeneracy' of the 3-form $T$.

The exterior derivative $dT$ has the following  expression (see e.g. \cite{I,IP2,FI})
\begin{equation}\label{dh}
\begin{split}
dT(X,Y,Z,V)=d^{\sb}T(X,Y,Z,V) +2\sigma^T(X,Y,Z,V), \quad where\\
d^{\sb}T(X,Y,Z,V)=(\nabla_XT)(Y,Z,V)+(\nabla_YT)(Z,X,V)+(\nabla_ZT)(X,Y,V)-(\nabla_VT)(X,Y,Z).
 \end{split}
 \end{equation}

 For the curvature of $\sb$ we use the convention $ R(X,Y)Z=[\nabla_X,\nabla_Y]Z -
 \nabla_{[X,Y]}Z$ and $ R(X,Y,Z,V)=g(R(X,Y)Z,V)$. It has the well-known properties
$% \begin{equation}\label{r1}
 R(X,Y,Z,V)=-R(Y,X,Z,V)=-R(X,Y,V,Z)
$.% \end{equation}

  The first Bianchi identity for $\nabla$ can be written in the  form (see e.g. \cite{I,IP2,FI})
 \begin{equation}\label{1bi}
 \begin{split}
 R(X,Y,Z,V)+ R(Y,Z,X,V)+ R(Z,X,Y,V)\\%=(\nabla_XH)(Y,Z,V)+(\nabla_YH)(Z,X,V)+(\nabla_ZH)(X,Y,V)+\sigma^T(X,Y,Z,V)\\
 =dT(X,Y,Z,V)-\sigma^T(X,Y,Z,V)+(\nabla_VT)(X,Y,Z).
 \end{split}
 \end{equation}
It is proved in \cite[p.307]{FI} that the curvature of  a metric connection $\sb$ with totally skew-symmetric torsion $T$  satisfies also the  identity
 \begin{equation}\label{gen}
 \begin{split}
 R(X,Y,Z,V)+ R(Y,Z,X,V)+ R(Z,X,Y,V)-R(V,X,Y,Z)-R(V,Y,Z,X)-R(V,Z,X,Y)\\
 =\frac32dT(X,Y,Z,V)-\sigma^T(X,Y,Z,V),
 \end{split}
 \end{equation}
 which combined with \eqref{1bi} yields \cite{IS, IS1}
 \begin{equation}\label{1bi1}
 \begin{split}
R(V,X,Y,Z)+R(V,Y,Z,X)+R(V,Z,X,Y)= -\frac12dT(X,Y,Z,V)+(\nabla_VT)(X,Y,Z)%+R(X,Y,Z,V)+ R(Y,Z,X,V)+ R(Z,X,Y,V)\\%=(\nabla_XH)(Y,Z,V)+(\nabla_YH)(Z,X,V)+(\nabla_ZH)(X,Y,V)+\sigma^T(X,Y,Z,V)\\
%\\ =dT(X,Y,Z,V)-\sigma^T(X,Y,Z,V)+(\nabla_VT)(X,Y,Z)-\frac32dT(X,Y,Z,V)+\sigma^T(X,Y,Z,V)
 \end{split}
 \end{equation}
%\begin{dfn} We say that the curvature $R$ satisfies the Riemannian first Bianchi identity if
%\begin{equation}\label{RB}
%R(X,Y,Z,V)+R(Y,Z,X,V)+R(Z,X,Y,V)=0.
%\end{equation}
%\end{dfn}
 %It is well known algebraic fact that \eqref{r1} and \eqref{RB} imply $R\in S^2\Lambda^2$, i.e.
% \begin{rmrk}
 %Note that, in general, \eqref{r1} and \eqref{r4} do not imply \eqref{RB}.
% \end{rmrk}
It is known that a metric connection $\sb$ with  torsion 3-form $T$ has curvature $R\in S^2\Lambda^2$, i.e. it satisfies 
 \begin{equation}\label{r4}
 R(X,Y,Z,V)=R(Z,V,X,Y)
 \end{equation}
 if and only if the covariant derivative of the torsion with respect to the torsion connection  is a 4-form.
\begin{lemma}\cite[Lemma~3.4]{I} The next equivalences hold for a metric connection with torsion 3-form
\begin{equation}\label{4form}
(\sb_XT)(Y,Z,V)=-(\sb_YT)(X,Z,V) \Longleftrightarrow R(X,Y,Z,V)=R(Z,V,X,Y) ) \Longleftrightarrow dT=4\LC T.
\end{equation}
\end{lemma}
%\begin{equation}\label{fourf}
%(\sb_XT)(Y,Z,V)=-(\sb_YT)(X,Z,V).
%\end{equation}
%We need the very recent result from \cite{IS}
%\begin{thrm} \cite[Theorem~1.2]{IS}\label{tFBI}
%A metric connection $\sb$ with torsion 3-form $T$ satisfies the Riemannian first Bianchi identity exactly when the next identities hold
%\begin{equation*}%\label{FBT}
% dT=-2\nabla T=\frac23\sigma^T.
%\end{equation*}
%In this case, the torsion 3-form $T$ is parallel with respect to a metric connection with torsion 3-form $\frac13T$ \cite{AF} and therefore has a constant norm, $||T||^2=const.$
%\end{thrm}

 The   Ricci tensors and scalar curvatures of the connections $\LC$ and $\sb$ are related by \cite[Section~2]{FI}, (see also \cite [Prop. 3.18]{GFS})
\begin{equation}\label{rics}
\begin{split}
Ric^g(X,Y)=Ric(X,Y)+\frac12 (\delta T)(X,Y)+\frac14\sum_{i=1}^n g(T(X,e_i),T(Y,e_i));\\
Scal^g=Scal+\frac14||T||^2,\qquad Ric(X,Y)-Ric(Y,X)=-(\delta T)(X,Y),
\end{split}
\end{equation}
where $\delta=(-1)^{np+n+1}*d*$ is the co-differential acting on $p$-forms and $*$ is the Hodge star operator satisfying $*^2=(-1)^{p(n-p)}$.

%One  has the general identities for $\alpha\in\Lambda^1$ and $\beta\in \Lambda^k$
%\begin{equation}\label{1star}
%\begin{split}
%*(\alpha\lrcorner\beta)=(-1)^{k+1}(\alpha\wedge*\beta);\qquad (\alpha\lrcorner\beta)=(-1)^{n(k+1)}*(\alpha\wedge*\beta);\\
%*(\alpha\lrcorner*\beta)=(-1)^{n(k+1)+1}(\alpha\wedge\beta);\qquad (\alpha\lrcorner*\beta)=(-1)^{k}*(\alpha\wedge\beta);
%\end{split}
%\end{equation}

\subsection{Torsion connection preserving an almost Hermitian structure}In this section we recall the notions and existence of a metric linear connection preserving a given almost Hermitian structure and having totally skew-symmetric torsion from \cite{FI}.

Let $(M,g,J)$ be an almost Hermitian 2n-dimensional manifold with Riemannian
metric $g$ and an almost complex structure $J$ satisfying
$$g(JX,JY)=g(X,Y).$$
The Nijenhuis tensor $N$, the K\"ahler
form $F$ and the Lee form $\theta$ are defined by
\begin{equation}\label{cy1}
N=[J.,J.]-[.,.]-J[J.,.]-J[.,J.], \quad F=g(.,J.), \quad \theta(.)=\delta F(J.),
\end{equation}
respectively, where $\delta=-*d*$ is the co-differential on an even dimensional manifold.% and * is the Hodge star operator.

Consequently, the  1-form $J\theta$ defined by $J\theta(X)=-\theta(JX)$ is co-closed due to $\delta(J\theta)=-\delta^2 F=0$.

\noindent The Nijenhuis tensor of type (0,3),  denoted by the same letter, is defined by $N(X,Y,Z)=g(N(X,Y),Z)$.

In general, on an almost Hermitian manifold, there not always exists a metric connection with totally skew-symmetric torsion preserving the almost Hermitian structure.
 It is shown in \cite[Theorem~10.1]{FI} that there
exists a unique linear connection $\sb$ preserving an almost hermitian
structure  and having totally skew-symmetric torsion $T$  if and only if
the Nijenhuis tensor $N$ is a 3-form.  In this case the 3-form$N$ is of type (3,0)+(0,3) with respect to the almost complex structure $J$. This is precisely the class $G_1$ in the Gray-Hervella classification of almost Hermitian manifolds \cite{GrH} and includes Hermitian manifolds, $N=0$ as well as nearly K\"ahler spaces, $(\LC_XJ)X=0$. We call this connection \emph{the torsion connection}.

The torsion $T$ of the torsion connection is
determined by \cite[Theorem~10.1]{FI}
\begin{equation}\label{cy2}
 T=JdF+N=-dF(J.,J.,J.)+N=-dF^+(J.,J.,J.)+\frac{1}{4}N,
\end{equation}
where $dF^+$ denotes the (1,2)+(2,1)-part of $dF$ with respect to $J$ and satisfies the equality
\begin{equation}\label{12}
\begin{split}
dF^+(X,Y,Z)=dF^+(JXJ,Y,Z)+dF^+(JX,Y,JZ)+dF^+(X,JY,JZ).
\end{split}
\end{equation}
The (3,0)+(0,3)-part $dF^-$ obeys the identity
$dF^-(JX,Y,Z)=dF^-(X,JY,Z)=dF^-(X,Y,JZ)$
and  is determined
completely by the Nijenhuis tensor \cite{gauduchon}.  If $N$ is a three form  then (see e.g.\cite{gauduchon,FI})
\begin{equation*}%\label{acy}
dF^-(X,Y,Z)=\frac{3}{4}N(JX,Y,Z).
\end{equation*}
The Lee form $\theta:=\delta F\circ J$ of the $G_1$- manifold is given in terms of the  torsion $T$ as
 \begin{equation}\label{liff}
\theta(X)=-\frac12T(JX,e_i,Je_i)=\frac12g(T(JX),F)=\frac12(F\lrcorner
T(JX))~,~~~\theta_i={1\over 2} J^k{}_i T_{kj\ell} F_{j\ell}~,
\end{equation}
where $\{e_1,\dots,e_n,e_{n+1}=Je_1,\dots,e_{2n}=Je_n\}$ is an orthonormal basis.

When the almost complex structure is integrable, $N=0$, the torsion connection was defined and studied by Strominger \cite{Str} in connection with the heterotic string background and was used by Bismut to prove a local index theorem for
the Dolbeault operator on Hermitian non-K\"ahler manifold \cite{bismut}. This formula was applied also in string theory (see e.g. \cite{BBE}). In the Hermitian case, the torsion connection is also
known as the Strominger-Bismut connection and has attained a lot of consideration both in mathematics and physics (see the Introduction).

In the nearly K\"ahler case,  the torsion connection $\sb$ is the characteristic connection
considered by Gray, \cite{gray} and the torsion $T$ and the Nijenhuis tensor are $\sb-$parallel, $\sb T=\sb N = 0$ (see e.g. \cite{Kir,BM}).

 In general $G_1$ manifold,  the holonomy group of $\sb$ is contained in $U(n)$. The reduction of the holonomy group of the torsion connection to a subgroup of $SU(n)$ can be expressed in terms of its Ricci 2-form which represents the first Chern class, namely
$$\rho(X,Y)=\frac12R(X,Y,e_i,Je_i)=0.
$$
%namely $$\rho=0.$$
We remark that the formula  \cite[(3.16)]{IP2} holds in the general case of a $G_1$-manifold (see also \cite{FI}), namely, the next formula holds true
\begin{equation*}%\label{ssc}
\rho(X,Y)=Ric(X,JY)+(\sb_X\theta)(JY)+\frac14dT(X,Y,e_i,Je_i).
\end{equation*}
Thus, on a $G_1$ manifold the restricted holonomy group of the torsion connection is contained in $SU(n)$ if and only if the next  condition holds \cite[Theorem~10.5]{FI}
\begin{equation*}%\label{su}
Ric(X,Y)+(\sb_X\theta)(Y)-\frac14dT(X,JY,e_i,Je_i)=0.
\end{equation*}
The scalar curvature of the torsion connection and the Riemannian scalar curvature on an ACYT space are given by \cite{IS1}
\begin{equation*}%\label{scal2}
\begin{split}
Scal=3\delta\theta+2||\theta||^2-\frac13||dF^+||^2+\frac1{16}||N||^2;\qquad%\\
Scal^g=3\delta\theta+2||\theta||^2-\frac13||dF^+||^2+\frac5{64}||N||^2.
\end{split}
\end{equation*}
It is shown in \cite{IS1} that if  the structure is balanced, $\theta=0$, and the torsion is closed, $dT=0$ then $16|dF|^2=3|N|^2$. In the complex case, $N=0$,  this confirms the old result first established in \cite{AI}  that  a balanced Hermitian manifold with closed torsion must be K\"ahler.

\section{SU(3)-structures}

Let $(M,g,J)$ be an almost Hermitian 6-manifold with a Riemannian
metric $g$ and an almost complex structure $J$, i.e. $(g,J)$
define an $U(3)$-structure.

An $SU(3)$-structure is determined by an additional  non-degenerate
(3,0)-form $\Psi=\Psi^++\sqrt{-1}\Psi^-$ of constant norm, or equivalently by a
non-trivial spinor. To be more explicit, we may choose a local
orthonormal frame  $e_1,\dots,e_6$,  identifying it  with the dual
basis via the metric. Write $e_{i_1 i_2\dots i_p}$ for the
monomial $e_{i_1} \wedge e_{i_2} \wedge \dots \wedge e_{i_p}$. An
$SU(3)$-structure is described locally by
\begin{equation}\label{AA}
\begin{split}
\Psi=-(e_1+\sqrt{-1}e_2)\wedge (e_3+\sqrt{-1}e_4)\wedge  (e_5+\sqrt{-1}e_6),\quad
F =-e_{12} - e_{34}- e_{56}, \\
\Psi^+ =-e_{135} + e_{236} + e_{146} +e_{245}, \quad
  \Psi^- =-e_{136} - e_{145} - e_{235} + e_{246}.
\end{split}
\end{equation}
These forms are subject to the compatibility relations
\begin{equation*}%\label{comp}
F\wedge\Psi^{\pm}=0, \qquad \Psi^+\wedge\Psi^-=-\frac23 F^3.
\end{equation*}
The subgroup of $SO(6)$ fixing the forms $F$ and  $\Psi$ simultaneously is $SU(3)$.
The two forms $F$ and $\Psi$ determine the metric completely. The Lie algebra of $SU(3)$
is denoted $\frak{su}(3)$.

The failure of the holonomy group of the Levi-Civita connection $\sb^g$ of the metric $g$ to be reduced to $SU(3)$
can be measured by the intrinsic torsion $\tau$, which is identified with $dF, d\ps$ and $d\sp$
and can be decomposed into five classes \cite{CSal}, $\tau \in W_1\oplus \dots \oplus W_5$. The
intrinsic torsion of an $U(n)$- structure belongs to the first four components described
by Gray-Hervella \cite{GrH}. The five components of a $SU(3)$-structure are first described
by Chiossi-Salamon \cite{CSal} (for interpretation in physics see \cite{Car,GLMW,GM}) and are
determined by $dF,d\Psi^+,d\Psi^-$ as well as $dF$ and $N$. We list below those of these classes which
we will use later.
\begin{description}
%\begin{enumerate}
\item[$\tau \in W_1:$] The class of Nearly K\"ahler (weak holonomy) manifold defined by
$dF$ to be (3,0)+(0,3)-form.
%\item[$\tau \in W_2$] The class of almost K\"ahler manifolds, $dF=0$.
\item[$\tau \in W_3:$] The class of balanced Hermitian manifold determined by the
conditions $N=\theta=0$.
%\item[$\tau \in W_4$] The class of locally conformally K\"ahler spaces characterized by
%$dF=\theta\wedge F$.
\item[$\tau \in W_1\oplus W_3\oplus W_4:$] The class called by Gray-Hervella $G_1$-manifolds
determined by the condition that the Nijenhuis tensor is totally skew-symmetric.
This is the precise class which we are interested in.
\item[$\tau \in W_1\oplus W_3:$] The class of balanced  $G_1$-manifolds
determined by the condition that the Nijenhuis tensor is totally skew-symmetric and the Lee form vanishes, $N(X,Y,Z)=-N(X,Z,Y),\quad \theta=0$.
\item[$\tau \in W_3\oplus W_4:$] The class of Hermitian manifolds, $N=0$.
%\end{enumerate}
\end{description}
If all five components are zero then we have a Ricci-flat K\"ahler (Calabi-Yau) 3-fold.

Let $(M,g,J,\Psi)$ be an $SU(3)$ manifold. Letting the four form $$\Phi=\frac12F^2=\frac12F\wedge F$$ we have the relations
\begin{equation}\label{star}
\begin{split}
*\Phi=-F,\quad *F=-\Phi,\qquad
*\Psi^+=\Psi^-,\quad *\Psi^-=-\Psi^+,\\
*(\alpha\wedge F)=-\alpha\lrcorner\Phi, \quad \alpha\in \Lambda^1,\qquad *(\alpha\wedge\Phi)=J\alpha=-\alpha\lrcorner F, \quad \alpha\in\Lambda^1,\\
*(\alpha\wedge\Psi^+)=-\alpha\lrcorner\Psi^- \quad \alpha\in \Lambda^1,\qquad
*(\alpha\wedge\Psi^-)=\alpha\lrcorner\Psi^+ \quad \alpha\in \Lambda^1,\\
*(\beta\wedge\ps)=\beta\lrcorner\sp, \quad \beta\in \Lambda^2,\qquad *(\beta\wedge\sp)=-\beta\lrcorner\ps, \quad \beta\in \Lambda^2.
\end{split}
\end{equation}
Extending the action of the almost complex structure on the exterior algebra, $J\beta=(-1)^p\beta(J.,\dots,J.)$ for a p-form $\beta\in\Lambda^p$, we recall the decomposition of the exterior algebra under the action of $SU(3)$ from \cite{BV} (see also \cite{LLR}). Let $V=T_pM$ be the vector space at a point $p\in M$ and $V^*$ be the dual vector space.  Obviously $SU(3)$
acts irreducibly on $V^*$ and $\Lambda^5V^*$, while $\Lambda^2V^*$ and $\Lambda^3V^*$ decompose as follows:
\begin{equation*}\label{dec}
\begin{split}
\Lambda^2V^*=\Lambda^2_1V^*\oplus\Lambda^2_6V^*\oplus\Lambda^2_8V^*,\\
\Lambda^3V^*=\Lambda^3_{Re}V^*\oplus\Lambda^3_{Im}V^*\oplus\Lambda^3_6V^*\oplus\Lambda^3_{12}V^*,
\end{split}
\end{equation*}
where
\begin{equation*}%\label{dec2}
\begin{split}
\Lambda^2_1V^*=\mathbb R.F, \qquad
\Lambda^2_6V^*=\{\phi\in \Lambda^2V^* | J\phi =-\phi\},\\
\Lambda^2_8V^*=\{\phi\in \Lambda^2V^* | J\phi=\phi, \phi\wedge F^2=0\}\\=\{\phi\in \Lambda^2V^* | *(\alpha\wedge F)=\alpha=\alpha\lrcorner *F=-\alpha\lrcorner\Phi\}\cong \frak{su}(3),\\
\Lambda^3_{Re}V^*=\mathbb R\Psi^+, \qquad \Lambda^3_{Im}V^*=\mathbb R\Psi^-,\\
\Lambda^3_6V^*=\{\alpha\wedge F  |  \alpha\in\Lambda^1V^*\},\\
\Lambda^3_{12}V^*=\{\gamma\in \Lambda^3V^* | \gamma\wedge F=\gamma\wedge\Psi^+=\gamma\wedge\Psi^-=0\}.
\end{split}
\end{equation*}
Writing
$$
F=\frac12F_{ij}e_{ij},\quad \Psi^+=\frac16\Psi^+_{ijk}e_{ijk}, \quad  \Psi^-=\frac16\Psi^-_{ijk}e_{ijk},\quad  \Phi=\frac1{24}\Phi_{ijkl}e_{ijkl},
$$
we have the obvious identites (c.f. \cite{BV})
\begin{equation}\label{iden}
\begin{split}
\Phi_{jslm}=F_{js}F_{lm}+F_{sl}F_{jm}+F_{lj}F_{sm},\\
\Psi^+_{ipq}F_{pq}=0=\Psi^-_{ipq}F_{pq}, \quad \Phi_{ijkl}\Psi^+_{jkl}=0=\Phi_{ijkl}\Psi^-_{jkl},\\
F_{ip}F_{pj}=-\delta_{ij},\quad \Psi^+_{ijs}F_{sk}=-\Psi^-_{ijk}, \quad \Psi^-_{ijs}F_{sk}=\Psi^+_{ijk},\\
\Psi^+_{ipq}\Psi^-_{jpq}=-4F_{ij}, \quad \Psi^+_{ipq}\Psi^+_{jpq}=4\delta_{ij}=\Psi^-_{ipq}\Psi^-_{jpq},\\
%\\\Psi^+_{kls}\Psi^-_{ijs}=\delta_{kj}F_{li}+\delta_{li}F_{kj}-\delta_{ki}F_{lj}-\delta_{lj}F_{ki},\\
%\Psi^+_{kls}\Psi^+_{ijs}=F_{kj}F_{li}-\delta_{li}\delta_{kj}-F_{ki}F_{lj}+\delta_{lj}\delta_{ki}\\
 \Phi_{ijkl}F_{kl}=4F_{ij},\quad \Phi_{ijkl}\Psi^+_{klp}=2\Psi^+_{ijp},\quad  \Phi_{ijkl}\Psi^-_{klp}=2\Psi^-_{ijp}.\\
% \Phi_{ijkl}\Psi^+_{lqp}=-F_{ij}\Psi^-_{pqk}-F_{jk}\sp_{pqi}-F_{ki}\sp_{pqj},\\
%\Phi_{ijkl}\Psi^-_{lqp}=F_{ij}\Psi^+_{pqk}+F_{jk}\ps_{pqi}+F_{ki}\ps_{pqj},\\
%\Phi_{ijkl}\Phi_{rjkl}=12\delta_{ir}, \quad  \Phi_{ijkl}\Phi_{klqr}=2F_{ij}F_{qr}-2\delta_{jq}\delta_{ir}+2\delta_{iq}\delta_{jr},\\
%\Phi_{ijkl}\Phi_{lpqr}=-F_{ij}F_{pq}\delta_{kr}-F_{ij}F_{qr}\delta_{kp}-F_{ij}F_{rp}\delta_{kq}\\
%-F_{jk}F_{pq}\delta_{ir}-F_{jk}F_{qr}\delta_{ip}-F_{jk}F_{rp}\delta_{iq}\\
%-F_{ki}F_{pq}\delta_{jr}-F_{ki}F_{qr}\delta_{jp}-F_{ki}F_{rp}\delta_{jq}.
\end{split}
\end{equation}
We need the following algebraic fact from \cite{IS1}
\begin{prop}\cite[Proposition~3.1]{IS1}\label{4-for}
Let $A$ be a 4-form and define the 3-form $B=(X\lrcorner A)$ for any $X\in T_pM$. If  the 3-form $B\in\Lambda^3_{12}$ then the four form $A$ vanishes identically, $A=0$.
\end{prop}

\section{Almost Calabi-Yau with torsion space  in dimension 6}
The presence of a parallel spinor with respect to a metric connection with torsion 3-form  in dimension6 leads
% to the reduction to $SU(3)$, i.e. 
to the existence of an almost
Hermitian structure and a linear
connection preserving the almost Hermitian structure with torsion
3-form %and thirdly to the reduction of the
with holonomy group inside %of the torsion connection to be a subgroup of
$SU(3)$.

The reduction of the holonomy group of the torsion connection to $SU(3)$, was investigated intensively in the complex case, i.e. when the induced almost complex structure is integrable ($N=0$) and it is known as the Strominger condition. The Strominger condition (see \cite{Str}) has different very useful expressions in terms of the five torsion classes (see e.g. \cite{Car,LY,yau,yau1, II}). We present here the results from \cite{II,IS1}. %which gives  necessary and sufficient condition to have a $\nabla-$parallel spinor in dimension 6,
\begin{thrm}\cite[Theorem~4.1]{II}\label{cythm1}
Let $(M,g,J,\Psi)$ be a 6-dimensional smooth manifold with an $SU(3)$-structure $(g,J,\Psi)$.
The next two conditions are equivalent:
\begin{enumerate}
\item[a)] $(M,g,J,\Psi)$ is an ACYT space, i.e. there exists a unique $SU(3)$-connection $\sb$ with torsion 3-form %, i.e. a
%linear connection with torsion 3-form %which preserves the almost hermitian structure
whose  holonomy is contained in SU(3),
%\be\label{consu3}
 $\sb g=\sb J=\sb \Psi^+=\sb\Psi^-=0$;
%\ee
\item[b)] The Nijenhuis tensor $N$ is totally skew-symmetric and the following two conditions
hold
\begin{equation}\label{cycon}
\begin{split}
d\Psi^+=\theta\wedge\Psi^+ -\frac{1}{4}(N,\Psi^+)* F; \qquad%\quad \delta\sp=-*d*\sp =*(\theta\wedge\ps)-\lambda F\\
d\Psi^-=\theta\wedge\Psi^- -\frac{1}{4}(N,\Psi^-)* F.%\quad \delta\ps=-*d*\ps=-*(\theta\wedge\sp)+\mu F.
\end{split}
\end{equation}
 The torsion is given by
\begin{equation}\label{torcy}
\begin{split}
T=-*dF +*(\theta\wedge F)+\frac{1}{4}(N,\Psi^+)\Psi^+ + \frac{1}{4}(N,\Psi^-)\Psi^-.
%\\0=d*T=d\theta\wedge F-\theta\wedge dF+\lambda d\sp-\mu d\ps\\
%=d\theta\wedge F-\theta\wedge dF +\lambda\Big[\theta\wedge\sp-\mu *F\Big]-\mu\Big[\theta\wedge\ps-\lambda *F\Big]\\
%d\theta\wedge F-\theta\wedge dF +\lambda\theta\wedge\sp-\mu\theta\wedge\ps=0.\\
%0=d\theta\wedge F\wedge F-\theta\wedge dF\wedge F=d\theta\wedge F\wedge F+\theta\wedge(*J\theta)=0.
\end{split}
\end{equation}
\end{enumerate}
%{\color{red}The Riemannian scalar curvature is expressed in the following way
%\begin{equation}\label{scal2}
%s^g = \frac{1}{8}(N,\Psi^+)^2 + \frac{1}{8}(N,\Psi^-)^2 + 2
%||\theta||^2 - \frac{1}{12}||T||^2 +3\delta\theta.
%\end{equation}
%}
%In particular, if the structure is complex and balanced then the Riemannian scalar
%curvature is non-positive.
\end{thrm}
In dimension 6 the (3,0)+(0,3) form $N$ is expressed in terms of the 3-forms $\ps$ and $\sp$ 
as follows $$N=\lambda\ps+\mu\sp,$$
where  the functions $\lambda,\mu$  are defined by the  scalar products
\begin{equation}\label{lm}
\lambda =\frac14(N,\ps)=\frac1{24}N_{ijk}\ps_{ijk}, \qquad \mu=\frac14(N,\sp)=\frac1{24}N_{ijk}\sp_{ijk}.
\end{equation}
The Lee form $\theta$ and the functions $\lambda,\mu$  can be expressed in terms of the torsion $T$, the 2-form $F$ and  the 4-form $\Phi$ as follows
\begin{equation}\label{tit}
\theta_i=\frac12T_{sjk}F_{jk}F_{si}=\frac16T_{jkl}\Phi_{jkli}; \quad \theta_iF_{iq}=-\frac12T_{qjk}F_{jk}; \quad \lambda=\frac16T_{klm}\ps_{klm};\quad \mu=\frac16T_{klm}\sp_{klm}.
\end{equation}
We also need the next observations from \cite{IS1}
\begin{thrm}\cite[Theorem~4.4]{IS1}\label{thnnew}
Let $(M,g,J,\Psi)$ be a 6-dimensional compact  ACYT manifold.%, i.e. a compact manifold with an $SU(3)$-structure  satisfying \eqref{cycon}.

Then  the  Nijenhuis tensor is parallel with respect to the torsion connection, $\sb N=0$, 
%and  the exterior derivative of the Lee form is $J$-invariant,
%$$\sb N=0, \qquad d\theta(X,Y)=d\theta(JX,JY).$$
 $d\lambda=d\mu=0$. In particular,  the Nijenhuis tensor has constant norm $d||N||^2=0$.
\end{thrm}
\begin{lemma}\cite[Lemma~4.3]{IS1}\label{lcom}
On a 6-dimensional ACYT manifold, the next identities hold
\begin{gather}\nonumber%\label{new2}
%\sb_i\mu=-\sb_p\lambda F_{pi}; \quad \sb_i\lambda=\sb_p\mu F_{pi},\quad \Longleftrightarrow
 d\mu=Jd\lambda, \quad d\lambda=-Jd\mu,\\
\label{new3}%{1.1}
d\theta(X,Y)-d\theta(JX,JY)=-(d\mu\lrcorner\sp)(X,Y)%;\qquad d\theta(X,Y)-\theta(JX,JY)
=-(d\lambda\lrcorner\ps)(X,Y).
\end{gather}
\end{lemma}

%Let $(M,g,J,\Psi)$ be a 6-dimensional  ACYT manifold. %, i.e. a smooth manifold with an $SU(3)$-structure satisfying \eqref{cycon}.
It is also known from \cite{IS1} that if on ACYT 6-manifold  the Lee form is closed, $d\theta=0$, then the Nijenhuis tensor is $\sb-$parallel, $\sb N=0$ and if the  Lee form vanishes, $\theta=0$ then  $\sb N=\delta T=0$   and the Ricci tensor of $\sb$ is symmetric.
%\[\sb N=\delta T=0,\qquad Ric(X,Y)=Ric(Y,X).\]

\section{Instanton conditions }
We begin this section with some general description of $\delta T$.
\subsection{Co-differential of the torsion}

We begin with the next
\begin{prop}\label{deltor}
On a 6-dimensional ACYT space the co-differential of the torsion is given by
\begin{equation}\label{tordel}
\delta T=*(d\theta\wedge F)-\theta\lrcorner T-2d\lambda\lrcorner\ps=-d\theta\lrcorner \Phi-\theta\lrcorner T-2d\lambda\lrcorner\ps.
\end{equation}
\end{prop}
\begin{proof}
We obtain from \eqref{torcy} 
\[*T=dF-\theta\wedge F+\lambda\sp-\mu\ps. \]
We calculate  using \eqref{cycon}  and \eqref{star} %the fact that if $d\lambda=d\mu=0$ then $d\theta\in\Lambda^2_8\cong su(3)$, 
that
\[
\begin{split}d*T=-d\theta\wedge F+\theta\wedge dF +d\lambda\wedge\sp-d\mu\wedge\ps+\lambda\theta\wedge\sp-\mu\theta\wedge\ps;\\
-\delta T=-*d*T=-*(d\theta\wedge F)+*(\theta\wedge dF)+*(d\lambda\wedge\sp)-*(d\mu\wedge\ps)+\lambda*(\theta\wedge\sp)-\mu*(\theta\wedge\ps)\\
=-*(d\theta\wedge F)-\theta\lrcorner *dF+(d\lambda+\lambda\theta)\lrcorner \ps +(d\mu+\mu\theta)\lrcorner\sp.
%\\\\=-d\theta+\theta\lrcorner(T-\theta\lrcorner *F-\lambda\ps-\mu\sp)+\lambda\theta\lrcorner \ps +\mu\theta\lrcorner\sp=-d\theta+\theta\lrcorner T=-d^{\sb}\theta.
\end{split}
\]
We express $*dF$ from \eqref{torcy} and use \eqref{new3} to write the above equality in the form
\[
\begin{split}
-\delta T=-*(d\theta\wedge F)+\theta\lrcorner(T-\theta\lrcorner *F-\lambda\ps-\mu\sp)+(d\lambda+\lambda\theta)\lrcorner \ps +(d\mu+\mu\theta)\lrcorner\sp.
%\\=-*(d\theta\wedge F)+\theta\lrcorner T+2d\lambda\lrcorner\ps
\end{split}
\]
Hence,  \eqref{tordel} follows.
\end{proof}
\begin{prop}\label{suin2}
On a 6-dimensional ACYT space with constant norm of the Nijenhuis tensor the 2-form $d\theta\in\Lambda^2_8\cong \frak{su}(3)$.
\end{prop}
\begin{proof}
Since $d\lambda=d\mu=0$, we have from \eqref{new3} that the 2-form $d\theta$ is of type (1,1) with respect to the almost complex structure $J$.  We have to show $d\theta\lrcorner F=0$.

Indeed, we have
\begin{equation}\label{ins7}
d^{\sb}\theta_{ij}F_{ij}=2\sb_i\theta_j F_{ij}=0
\end{equation}
since $\delta J\theta=\sb_i\theta_jF_{ij}=\d^2 F=0$.

It follows from \eqref{tsym} that $d\theta=d^{\sb}\theta+\theta\lrcorner T$. The definition of the Lee form, \eqref{liff} and \eqref{ins7} yield
\[d\theta_{ij}F_{ij}=d^{\sb}\theta_{ij}F_{ij}+\theta_sT_{sij}F_{ij}= \theta(J\theta)=0.\]
\end{proof}
\begin{prop}\label{suin1}
On a 6-dimensional ACYT space with constant norm of the Nijenhuis tensor we have 
\begin{equation}\label{delT}
\delta T=d^{\sb}\theta.
\end{equation}
\end{prop} 
\begin{proof}
By Proposition~\ref{suin2} we know $d\theta\in\Lambda^2_8\cong \frak{su}(3)$ which combined with  $d\lambda=0$ help us to write  \eqref{tordel} in the form $\delta T=d\theta-\theta\lrcorner T=d^{\sb}\theta$.
\end{proof}

\subsection{$SU(3)$--Instanton}
 Since the torsion connection $\sb$ preserves the $SU(3)$-structure its curvature  lies in the Lie algebra $\frak{su}(3)$, i.e. it satisfies
\begin{equation}\label{rr}
\begin{split}
R_{ijab}\ps_{abc}=R_{ijab}\sp_{abc}=R_{ijab}F_{ab}=0 \Longleftrightarrow R_{ijab}\Phi_{abkl}=-2R_{ijkl}.
%R(X,Y,e_i,e_j)\ps(e_i,e_j,Z)=R(X,Y,e_i,e_j)\sp(e_i,e_j,Z)=R(X,Y,e_i,e_j)F(e_i,e_j)=0 \Longleftrightarrow R(X,Y,e_i,e_j)\Phi(e_i,e_j,Z,V)=-2R(X,Y,Z,V).
%.\\
%R_{ijab}p_{abk}=0 \Longleftrightarrow R_{ijab}ph_{abkl}=-2R_{ijkl}.
\end{split}
\end{equation}
The  $SU(3)$--instanton condition means that the curvature 2-form of $\sb$ lies in the Lie algebra $\frak{su}(3)$. That is it is of type (1,1) with zero trace, 
\begin{equation}\label{instt2}
R(JX,JY,Z,V)=R(X,Y,Z,V),\qquad R(e_i,Je_i,Z,V)=0.
\end{equation}
In other words, the $SU(3)$--instanton condition  can be written in the form
\begin{equation}\label{rr1}
\begin{split}
R_{abij}\ps_{abc}=R_{abij}\sp_{abc}=R_{abij}F_{ab}=0 \Longleftrightarrow R_{abij}\Phi_{abkl}=-2R_{klij}.
\end{split}
\end{equation}
%which proves \eqref{dtt}.
%\end{proof}
%{\color{blue} 
\begin{lemma}\label{insu3}
On a 6 dimensional ACYT space with  constant norm of the Nijenhuis tensor if the  curvature of the torsion connection $\sb$   is an $SU(3)$-instanton then  $\delta T=d^{\sb}\theta\in\Lambda^2_{8}\cong  \frak{su}(3)$.% the $G_2$ structure is of constant type and the torsion is closed, $d\lambda=dT=0$.
\end{lemma}
\begin{proof}
We express the Ricci tensor in two ways using \eqref{rr} ,\eqref{1bi1}, \eqref{1bi}, \eqref{tit} and \eqref{dh}  as follows:
\begin{multline}\label{ricdt}
\begin{split}
Ric_{ij}=\frac12R_{iabc}\Phi_{jabc}=\frac16\Big[R_{iabc}+R_{ibca}+R_{icab} \Big]\Phi_{jabc}=\frac1{12}dT_{iabc}\Phi_{jabc}+\frac16\sb_iT_{abc}\Phi_{jabc}\\
=\frac1{12}dT_{iabc}\Phi_{jabc}-\sb_i\theta_j;\\
-Ric_{ji}=\frac12R_{abci}\Phi_{jabc}=\frac16\Big[ R_{abci}+R_{bcai}+R_{cabi} \Big]\Phi_{jabc}=\frac16\Big[dT_{abci}-\sigma^T_{abci}+\sb_iT_{abc}\Big]\Phi_{jabc}\\
=-\frac16dT_{iabc}\Phi_{jabc}+\frac16\sigma^T_{iabc}\Phi_{jabc}-\sb_i\theta_j.
\end{split}
\end{multline}
Taking the sum of the two equations in \eqref{ricdt}  and  using \eqref{rics} and \eqref{delT}, yields
\begin{equation}\label{ins4}
\begin{split}
-d^{\sb}\theta_{ij}=-\delta T_{ij}=Ric_{ij}-Ric_{ji}=\frac16\Big[-\frac12dT_{iabc}+\sigma^T_{iabc}\Big]\Phi_{jabc}-2\sb_i\theta_j\\
=-\frac1{12}d^{\sb}T_{abci}\Phi_{abcj}-2\sb_i\theta_j=-\frac14\sb_aT_{bci}\Phi_{abcj}-\frac32\sb_i\theta_j.
\end{split}
\end{equation}
With the help of \eqref{iden} and \eqref{liff}, we calculate
\begin{equation}\label{ins5}
\begin{split}
\sb_aT_{bci}\Phi_{abcj}=\sb_aT_{bci}\Big[F_{ab}F_{cj}+F_{bc}F_{aj}+F_{ca}F_{bj}\Big]\\
=(\sb_{Je_b}T)(e_b,Je_j,e_i)-(\sb_{Je_c}T)(Je_j,e_c,e_i)+(\sb_{Je_j}T)(Je_c,e_c,e_i)\\=2(\sb_{e_b}T)(Je_b,e_i,Je_j)-2(\sb_{Je_j}\theta)(Je_i).
\end{split}
\end{equation}
It is known due to \cite[(3.38)]{I} that 
\begin{multline}\label{inst1}
2R(X,Y,Z,V)-2R(Z,V,X,Y)\\=(\sb_XT)(Y,Z,V)-(\sb_YT)(X,Z,V)-(\sb_ZT)(X,Y,V)+(\sb_VT)(X,Y,Z).
\end{multline}
We apply the  $SU(3)$--instanton condition  to \eqref{inst1}. 
The equalities \eqref{rr} and \eqref{rr1} imply with the help of \eqref{liff}, \eqref{tit} and the first identity in \eqref{iden} that
\begin{multline}\label{inst2}
0=2R(X,JY,e_i,Je_i)-2R(e_i,Je_i,X,JY)\\=(\sb_XT)(JY,e_i,Je_i)-(\sb_{JY}T)(X,e_i,Je_i)-(\sb_{e_i}T)(Je_i,X,JY)+(\sb_{Je_i}T)(e_i,X,JY)\\
=-2(\sb_X\theta)Y-2(\sb_JY\theta)JX-2(\sb_{e_i}T)(Je_i,X,JY).%=-2(\sb_{e_i}T)(Je_i,X,Y).
\end{multline}
Combining \eqref{ins5}, \eqref{inst2} and \eqref{ins4}, we obtain
\begin{equation}\label{ins6}
\begin{split}
-d^{\sb}\theta(X,Y)=-(\sb_X\theta)Y+(\sb_Y\theta)X=-\frac32(\sb_X\theta)Y+\frac12(\sb_{JY}\theta)JX-\frac12(\sb_{e_i}T)(Je_i,X,JY)\\=
-\frac32(\sb_X\theta)Y+\frac12(\sb_{JY}\theta)JX+\frac12(\sb_X\theta)Y+\frac12(\sb_{JY}\theta)JX.
\end{split}
\end{equation}
Hence, \eqref{ins6} yields $(\sb_Y\theta)X=(\sb_{JY}\theta)JX$ and $d^{\sb}\theta(Y,X)=d^{\sb}\theta(JY,JX)$ which combined with \eqref{ins7} completes the proof.
\end{proof}
Proposition~\ref{suin1} and Lemma~\ref{insu3} imply
\begin{cor}\label{insu4}
On a 6 dimensional ACYT space with  constant norm of the Nijenhuis tensor.  If the  curvature of the torsion connection $\sb$   is an $SU(3)$-instanton then  $\theta\lrcorner T\in\Lambda^2_{8}\cong  \frak{su}(3)$.% the $G_2$ structure is of constant type and the torsion is closed, $d\lambda=dT=0$.

In particular, 
\begin{equation}\label{tetor}
T(\theta,X,Y)=T(\theta,JX,JY).
\end{equation}
\end{cor}
%}
\begin{thrm}\label{dt0}
Let $(M,g,J,\Psi)$ be a 6-dimensional  ACYT manifold with  constant norm of the Nijenhuis tensor  and the  curvature of the torsion connection $\sb$   is an $SU(3)$-instanton. %with instanton curvature. 

If the Lee form is $\sb$-parallel then $d^{\sb}T=0$.
\end{thrm}
\begin{proof}
We calculate from \eqref{dh}
\begin{multline}\label{inst3}
d^{\sb}T_{pjkl}\Phi_{jkli}=\sb_pT_{jkl}\Phi_{jkli}-3\sb_lT_{jkp}\Phi_{jkli}=6\sb_p\theta_i-3\sb_lT_{jkp}\Big[F_{jk}F_{li}+F_{kl}F_{ji}+F_{lj}F_{ki}\Big]\\=
6\sb_p\theta_i+6\sb_l\theta_sF_{sp}F_{li}-6\sb_lT_{kjp}F_{lk}F_{ji}=0,
\end{multline}
where we used $\sb\theta=0$ and \eqref{inst2} to get the last equality.

We have from \eqref{rr} using \eqref{1bi1}, \eqref{tit} and \eqref{dh}   the following %Ricci tensor $Ric$ of  $\sb$ is given by
%\begin{equation}\label{ricdt}
%Ric_{ij}=\frac12R_{iabc}\Phi_{jabc}=\frac16\Big[R_{iabc}+R_{ibca}+R_{icab} \Big]\Phi_{jabc}%=\frac1{12}dT_{iabc}\Phi_{jabc}+\frac16\sb_iT_{abc}\Phi_{jabc}
%=\frac1{12}dT_{iabc}\Phi_{jabc}-\sb_i\theta_j.
%\end{equation}
%This completes the proof of the first identity in \eqref{ricg2}.
%Similarly, we have
\begin{equation}\label{ricnew}
\begin{split}0=R_{iabc}\ps_{abc}=\frac13\Big[R_{iabc}+R_{ibca}+R_{icab} \Big]\ps_{abc}%=\frac16dT_{iabc}\ps_{abc}+\frac13\sb_iT_{abc}\ps_{abc}
=\frac16dT_{iabc}\ps_{abc}+2\sb_i\lambda,\\
0=R_{iabc}\sp_{abc}=\frac13\Big[R_{iabc}+R_{ibca}+R_{icab} \Big]\sp_{abc}%=\frac16dT_{iabc}\sp_{abc}+\frac13\sb_iT_{abc}\sp_{abc}
=\frac16dT_{iabc}\sp_{abc}+2\sb_i\mu.\end{split}
\end{equation}

Since $M$ is of constant type then $d\lambda=d\mu=0$ and we get from \eqref{ricnew} that 
\begin{equation}\label{inst4}
dT_{iabc}\ps_{abc}=dT_{iabc}\sp_{abc}=0.
\end{equation}
The  $SU(3)$--instanton condition \eqref{rr1} together with \eqref{rr} applied to \eqref{gen} yield
\begin{equation}\label{inst5}
0=\Big[R_{abci}+R_{bcai}+R_{cabi}-R_{iabc}-R_{ibca}-R_{icab}  \Big]\Psi^{\pm}_{abc}=\Big[\frac32dT_{abci}-\sigma^T_{abci}\Big]\Psi^{\pm}_{abc}=-\sigma^T_{abci}\Psi^{\pm}_{abc},
\end{equation}
where we used \eqref{inst4} to achieve the last equality.

Now, \eqref{dh} together with \eqref{inst4} and \eqref{inst5} implies
\begin{equation}\label{inst6}
0=dT_{iabc}\Psi^{\pm}_{abc}=d^{\sb}T_{iabc}\Psi^{\pm}_{abc}+2\sigma^T_{iabc}\Psi^{\pm}_{abc}=d^{\sb}T_{iabc}\Psi^{\pm}_{abc}.
\end{equation}
The identities \eqref{inst3} and \eqref{inst6} show that for any vector field $X$ the 3-form $(X\lrcorner d^{\sb}T)\in \Lambda^3_{12}$. Then the 4-form $d^{\sb}T=0$ due to Proposition~\ref{4-for}.
\end{proof}
In view of $d^{\sb}T=0$ we get from  \eqref{dh} 
\begin{equation}\label{inst8}
dT=2\sigma^T
\end{equation}
and we can write \eqref{inst1} in the form
\begin{multline}\label{inst7}
R(X,Y,Z,V)-R(Z,V,X,Y)\\=(\sb_XT)(Y,Z,V)-(\sb_YT)(X,Z,V)=-(\sb_ZT)(X,Y,V)+(\sb_VT)(X,Y,Z).
\end{multline}
Substitute \eqref{inst8} into \eqref{1bi} and  into \eqref{1bi1} to get
\begin{equation}\label{inst9}
\begin{split}
R(X,Y,Z,V)+R(Y,Z,X,V)+R(Z,X,Y,V)=\sigma^T(X,Y,Z,V)+(\sb_VT)(X,Y,Z);\\
R(V,X,Y,Z)+R(V,Y,Z,X)+R(V,Z,X,Y)=-\sigma^T(X,Y,Z,V)+(\sb_VT)(X,Y,Z).
\end{split}
\end{equation}
The sum of the two identities in \eqref{inst9} yields
\begin{multline}\label{inst10}
R(X,Y,Z,V)+R(Y,Z,X,V)+R(Z,X,Y,V)+R(V,X,Y,Z)+R(V,Y,Z,X)+R(V,Z,X,Y)\\=2(\sb_VT)(X,Y,Z).
\end{multline}
Using the fact that the curvature is (1,1) with respect to the first and the second pairs, we get from \eqref{inst10}
\begin{equation}\label{inst11}
(\sb_VT)(Z,JX,JY)=(\sb_{JV}T)(JZ,X,Y).
\end{equation}

Applying \eqref{rr} and \eqref{instt2} to \eqref{inst7}, we obtain 
\begin{equation}\label{inst12}
(\sb_XT)(Y,Z,V)-(\sb_XT)(Y,JZ,JV)=(\sb_YT)(X,Z,V)-(\sb_YT)(X,JZ,JV).
\end{equation}
Define the tensor $C$ as follows
\begin{equation}\label{cht}
2C(Z;X,Y)=T^+(Z,JX,JY)-T^+(Z,X,Y),
\end{equation}
where $T^+$ denotes the (1,2)+(2,1)-part of the torsion 3-form $T$. 

From  equation \eqref{cy2} we get  $T^+=-dF^+(J,J,J),\quad T^-=\frac14N$.
It is easy to check applying \eqref{12} 
\begin{equation}\label{tch}
\begin{split}
T^+(Z,X,Y)=-C(Z;X,Y)-C(X;Y,Z)-C(Y;Z,X); \\ C(JZ;X,Y)=-C(Z;JX,Y)=-C(Z;X,JY).
\end{split}
\end{equation}
Note that in the complex case, $N=0$, the tensor $C$ is precisely the torsion of the Chern connection.

It follows from \eqref{cht} and \eqref{tch} that the condition $\sb C=0$ is equivalent to  $\sb T^+=0$. Since the Nijenhuis tensor $N=4T^-$ is $\sb$-parallel  we get  that $\sb C=0$ exactly when $ \sb T=0$.
%then 
%\begin{equation}\label{inst15}
%(\sb_VC)(JZ;X,Y)=-(\sb_VC)(Z;JX,Y)=-(\sb_VC)(Z;X,JY).
%\end{equation}

Now, equality \eqref{inst11} yields
\begin{equation}\label{inst14}
(\sb_VC)(Z;X,Y)=-(\sb_{JV}C)(JZ;X,Y).
\end{equation}
We can write \eqref{inst12} in the form
\begin{equation}\label{inst13}
(\sb_XC)(Y;Z,V)=(\sb_YC)(X;Z,V).
\end{equation}
We use the complex basis $E_{\alpha}=e_a-iJe_a,\quad \bar{E_a}=E_{\bar{\alpha}}=e_a+iJe_a$.

We obtain from \eqref{tch}, \eqref{inst14} and \eqref{inst13} that 
\begin{equation}\label{propch}
\begin{split}
C_{\bar\alpha;\bar\beta\gamma}=C_{\bar\alpha;\bar\beta\bar\gamma}=C_{\alpha\beta\gamma}=C_{\alpha;\beta\bar\gamma}=0 \,,\quad T^+_{\bar\alpha\beta\gamma}=-T^+_{\beta\bar\alpha\gamma}=-C_{\bar\alpha;\beta\gamma}=C_{\bar\alpha;\gamma\beta} \,,\\
\sb_{\alpha}C_{\bar\beta;\gamma\mu}=\sb_{\bar\alpha}C_{\beta;\bar\gamma\bar\mu}=0 \,, \quad \sb_{\alpha}C_{\beta;\bar\gamma\bar\mu}=\sb_{\beta}C_{\alpha;\bar\gamma\bar\mu} \,, \quad  \sb_{\bar\alpha}C_{\bar\beta;\gamma\mu}=\sb_{\bar\beta}C_{\bar\alpha;\gamma\mu} \,.
\end{split}
\end{equation}
\begin{lemma}\label{instl} 
If $\sb$ is an $SU(3)$--instanton connection with parallel Lee form, then it holds
\[
\begin{split}
\sb_{\bar\alpha}\sb_{\bar\beta}C_{\bar\gamma;\mu\nu}-\sb_{\bar\beta}\sb_{\bar\alpha}C_{\bar\gamma;\mu\nu}=-T^+_{\bar\alpha\bar\beta\sigma}\sb_{\bar\sigma}C_{\bar\gamma;\mu\nu}; \\ 2\sb_{\bar\alpha}\sb_{\bar\beta}C_{\bar\gamma;\alpha\beta}=-T^+_{\bar\alpha\bar\beta\sigma}\sb_{\bar\gamma}C_{\bar\sigma;\alpha\beta}=\frac12\sb_{\bar\gamma}||C||^2; \\ \sb_{\bar\alpha}\sb_{\bar\beta}C_{\bar\gamma;\alpha\beta}=0.
\end{split}
\]
In particular, the norm of the torsion is a constant.
\end{lemma}
\begin{proof}
The instanton condition yields the curvature 2-form $R_{\bar a\bar b}=0$. Therefore,
\[0=R_{\bar\alpha\bar\beta}\times C_{\bar\gamma;\mu\nu}=R_{\bar\alpha,\bar\beta,\bar\gamma,\sigma}C_{\bar\sigma;\mu\nu}+R_{\bar\alpha,\bar\beta,\mu,\sigma}C_{\bar\gamma;\sigma\nu}+R_{\bar\alpha,\bar\beta,\nu,\sigma}C_{\bar\gamma;\mu\sigma}.\]
The Ricci identity for $\sb$ reads with the help of \eqref{cy2} 
\[\sb_{\bar\alpha}\sb_{\bar\beta}C_{\bar\gamma;\mu\nu}-\sb_{\bar\beta}\sb_{\bar\alpha}C_{\bar\gamma;\mu\nu}=-R_{\bar\alpha\bar\beta}\times C_{\bar\gamma;\mu\nu}-T^+_{\bar\alpha\bar\beta\sigma}\sb_{\bar\sigma}C_{\bar\gamma;\mu\nu}%=-T^+_{\bar\alpha\bar\beta\sigma}\sb_{\bar\sigma}C_{\bar\gamma;\mu\nu} %-T^+_{\bar a\bar b s}\sb_{\bar s}C_{\bar c;lm}
-\frac14N_{\bar\alpha\bar\beta\bar\sigma}\sb_{\sigma}C_{\bar\gamma;\mu\nu}=-T^+_{\bar\alpha\bar\beta\sigma}\sb_{\bar\sigma}C_{\bar\gamma;\mu\nu} \,, %T^+_{\bar a\bar b s}\sb_{\bar s}C_{\bar c;lm}
\]
since  $\sb_{\sigma}C_{\bar\gamma;\mu\nu}=0$ because of \eqref{propch}. This proves the  first line of the lemma.

Set $\mu=\alpha, \nu=\beta$ into the already proved first equality  we obtain the second line using  \eqref{propch}.

To show the last statement, we apply again \eqref{inst13} to get
\[ \sb_{\bar\beta}\sb_{\bar\alpha}C_{\bar\gamma;\beta\alpha}=\sb_{\bar\beta}\sb_{\bar\gamma}C_{\bar\alpha;\beta\alpha}=\sb_{\bar\beta}\sb_{\bar\gamma}\theta_{\beta}=0\]
which combined with the already proved second identity of the lemma yields $\sb||C||^2=0$ and therefore $\sb||T||^2=0$ since $\sb N=0$.
\end{proof}
\subsection{Proof of Theorem~\ref{main}}
It is clear from \eqref{rr} that if \eqref{r4} holds, i.e. $\sb T$ is a 4-form due to \eqref{4form}, then   the curvature of the torsion connection $\sb$ of an ACYT manifold is an instanton.  In particular, if $\sb T=0$ then $R$ is an instanton, the Lee form is $\sb$-parallel and therefore co-closed, $\delta\theta=0$. The next result completes  the proof of Theorem~\ref{main}.% we show below that if the ACYT space is balanced, of constant type and $R$ is an instanton then the torsion is $\sb$-parallel, $\sb T=0$ and \eqref{rr} holds.
\begin{thrm}\label{main1}  Let $(M,g,J,\Psi)$ be an ACYT 6-manifold  with  constant norm of the Nijenhuis tensor and $\sb$-parallel Lee form.  If the  curvature of the torsion connection $\sb$   is an $SU(3)$-instanton   then the torsion is $\sb$-parallel, $\sb T=0$.
%Let $(M,g,J,\Psi)$ be a compact balanced ACYT 6-manifold  and $\sb$ be an instanton connection.  Then the torsion is $\sb$-parallel, $\sb T=0$.
\end{thrm}
\begin{proof} 
Since $||C||^2=const$, we have
\begin{multline}\label{inst17}
0=\sb_{\alpha}\sb_{\bar\alpha}||C||^2=\sb_{\alpha}\Big[\sb_{\bar\alpha}C_{\bar\sigma;\mu\nu}C_{\sigma;\bar\mu\bar\nu}+C_{\bar\sigma;\mu\nu}\sb_{\bar\alpha}C_{\sigma;\bar\mu\bar\nu}\Big]=\sb_{\alpha}\sb_{\bar\alpha}C_{\bar\sigma;\mu\nu}C_{\sigma;\bar\mu\bar\nu}+||\sb C||^2,
\end{multline}
where we apply \eqref{inst14} to see that $\sb_{\bar\alpha}C_{\sigma;\bar\mu\bar\nu}=0$.

The Ricci identity together with \eqref{inst14} yields 
\begin{equation}\label{ricid}
\begin{split}
\sb_{\alpha}\sb_{\bar\alpha}C_{\bar\gamma;\mu\nu}=\sb_{\bar\alpha}\sb_{ \alpha}C_{\bar\gamma;\mu\nu}-R_{\alpha\bar\alpha}\times C_{\bar\gamma;\mu\nu}-T_{\alpha\bar 
\alpha\sigma}\sb_{\bar\sigma}C_{\bar\gamma;\mu\nu}- T_{\alpha\bar 
\alpha\bar\sigma}\sb_{\sigma}C_{\bar\gamma;\mu\nu}\\ =2\theta_{\sigma}\sb_{\bar\sigma}C_{\bar\gamma;\mu\nu}=2\theta_{\sigma}\sb_{\bar\gamma}C_{\bar\sigma;\mu\nu}=2\sb_{\bar\gamma}(\theta_{\sigma}C_{\bar\sigma;\mu\nu}),
\end{split}
\end{equation}
 where the first and the fourth terms in the right hand side vanishes because of \eqref{inst14}, the second term is zero since $R$ is an instanton, $R_{\alpha\bar\alpha}=0$, and for the third term we used the identity $T_{\alpha\bar\alpha\gamma}=-2\theta_{\gamma}$. The last two equalities  we obtain because of the symmetries in \eqref{propch} and the condition $\sb\theta=0$.

Using Lemma~\ref{insu4},  \eqref{cy2} and \eqref{cht},  we obtain
\[ 2C(\theta;X,Y)=T^+(\theta,JX,JY)-T^+T(\theta,X,Y)=T(\theta,JX,JY)-T(\theta,X,Y)-\frac12N(\theta,X,Y)=-\frac12N(\theta,X,Y),\]
which yields 
\begin{equation}\label{par}
(\sb_ZC)(\theta;X,Y)=0, \quad since \quad \sb N=\sb\theta=0.
\end{equation}
Now, \eqref{par}, \eqref{ricid} and \eqref{inst17} imply $||\sb T||^2=||\sb C||^2=0$ which completes the proof of Theorem~\ref{main}.
\end{proof}

\subsection{Proof of Theorem~\ref{mainc}}
It is known from \cite[Theorem~2.4]{IS1} that on any compact 6-dimensional ACYT the Nijenhuis tensor is $\sb$-parallel.  In this case we may replace the condition $\sb\theta=0$ in Theorem~\ref{main} with the weaker condition $\delta\theta=0$ due to the next result
\begin{thrm}\label{gaud}
 Let $(M,g,J,\Psi)$ be a compact ACYT 6-manifold  with  co-closed Lee form, $\delta\theta=0$.  If the  curvature of the torsion connection $\sb$   is an $SU(3)$-instanton   then the Lee form is $\sb$-parallel. %the torsion is $\sb$-parallel, $\sb T=0$.}
\end{thrm}
\begin{proof} We start with the next identity
\begin{equation}\label{iii}
\sb_i\delta T_{ij}=\frac12\delta T_{ia}T_{iaj} \,,
\end{equation}
 shown in \cite[Proposition~3.2]{IS} for any metric connection with a totally skew-symmetric torsion. %We give here a proof  of \eqref{iii} for completeness.  The identity $\delta^2=0$ together with \eqref{tsym}  imply 
%$$0=\delta^2 T_k=\LC_i\LC_jT_{ijk}=\LC_i\sb_jT_{ijk}=\sb_i\delta T_{ik}+\frac12T_{iks}\delta T_{is}=\sb_i\delta T_{ik}-\frac12\delta T_{is}T_{isk}.$$

We calculate the left-hand side  of \eqref{iii} applying Lemma~\ref{insu3} as follows
\begin{equation}\label{ntss}
\begin{split}
\sb_i\delta T_{ij}=\sb_id^{\sb}\theta_{ij}=%\sb_i\sb_i\theta_j-\sb_i\sb_j\theta_i-\frac12(\sb_i\sb_t-\sb_t\sb_i)\lambda\p_{tij}\\=
\sb_i\sb_i\theta_j-\sb_i\sb_j\theta_i.%-\frac12T_{tis}\sb_s\lambda\p_{tij},
\end{split}
\end{equation}
%where we applied $d^2\lambda=0$ and \eqref{tsym} to get the last term.
Substitute \eqref{ntss} into \eqref{iii} %using \eqref{deltaT} 
to get
\begin{equation}\label{ntss1}
\sb_i\sb_i\theta_j-\sb_i\sb_j\theta_i=\frac12d^{\sb}\theta_{ab}T_{abj}.
\end{equation}
The Ricci identity  
$\sb_i\sb_j\theta_i=%\sb_j\sb_i\theta_i-R_{ijis}\theta_s-T_{ija}\sb_a\theta_i=
\sb_j\sb_i\theta_i+Ric_{js}\theta_s-\frac12d^{\sb}\theta_{ai}T_{aij}$ substituted into \eqref{ntss1} yields
\begin{equation}\label{gafin}
\sb_i\sb_i\theta_j+\sb_j\delta\theta-Ric_{js}\theta_s=0.
\end{equation}
Multiply the both sides of \eqref{gafin} with $\theta_j$, use  $\delta\theta=0$ together with the identity 
$$\frac12\Delta||\theta|^2=-\frac12\LC_i\LC_i||\theta||^2=-\frac12\sb_i\sb_i||\theta||^2=-\theta_j\sb_i\sb_i\theta_j-||\sb\theta||^2$$ 
to get
\begin{equation}\label{gafinf}
-\frac12\Delta||\theta||^2-Ric(\theta,\theta)-||\sb\theta||^2=0.
\end{equation}
We obtain from \eqref{ins4} that 
\begin{equation}\label{rdt1}
\sb_aT_{bci}\ph_{abcj}=-2\sb_i\theta_j-4\sb_j\theta_i.
\end{equation}
We get from \eqref{ricdt} applying \eqref{dh} and \eqref{rdt1}
\begin{equation}\label{rdt}
\begin{split}
Ric_{ij}\theta_i\theta_j=\frac1{12}dT_{abci}\ph_{abcj}\theta_i\theta_j-\theta_i\theta_j\sb_i\theta_j\\=
\frac1{12}\Big[3\sb_aT_{bci}\ph_{abcj}-6\sb_i\theta_j +2\sigma^T_{abci}\ph_{abcj} \Big]\theta_i\theta_j-\theta_i\theta_j\sb_i\theta_j=-\frac32\theta_i\sb_i||\theta||^2 +\frac16\sigma^T_{abci}\ph_{abcj} \theta_i\theta_j.
\end{split}
\end{equation}
We will show that the last term in \eqref{rdt} vanishes.

Indeed, we have from Corollary~\ref{insu4} that  $\theta\lrcorner T=d\theta-\delta T\in \Lambda^2_8\cong\frak{su}(3)$.% yielding $\sigma^T_{abci}\ph_{abcj}\theta_i\theta_j=0$.

We calculate from \eqref{sigma} applying the first identity in \eqref{iden} and \eqref{tetor}
\begin{multline}\label{pres}
\frac13\sigma^T(e_a,e_b,e_c,\theta)\ph(e_a,e_b,e_c,\theta)=T(e_a,e_b,e_s)T(e_s,e_c,\theta)\ph(e_a,e_b,e_c,\theta)\\=T(e_a,e_b,e_s)T(e_s,e_c,\theta)\Big[F(e_a,e_b)F(e_c,\theta)+F(e_b,e_c)F(e_a,\theta)+F(e_c,e_a)F(e_b,\theta)    \Big]\\=T(Je_b,e_b,e_s)T(e_s,J\theta,\theta)+T(J\theta,Je_c,e_s)T(e_s,e_c,\theta)-T(Je_c,J\theta,e_s)T(e_s,e_c,\theta)\\=-T(Je_b,e_b,e_s)T(Je_s,\theta,\theta)+2T(J\theta,Je_c,e_s)T(e_s,e_c,\theta) =2T(J\theta,Je_c,e_s)T(e_s,e_c,\theta).
\end{multline}
Applying  \eqref{tetor}, we have
\[
\begin{split}T(J\theta,Je_c,e_s)T(e_s,e_c,\theta) =-T(J\theta,e_c,e_s)T(e_s,Je_c,\theta)=T(J\theta,e_c,e_s)T(Je_s,e_c,\theta)\\
=-T(J\theta,e_c,Je_s)T(e_s,e_c,\theta)=-T(J\theta,Je_s,e_c)T(e_s,e_c,\theta).
\end{split}\]
Hence, $T(J\theta,Je_c,e_s)T(e_s,e_c,\theta)=0$ and  \eqref{pres} yields $\sigma^T_{abci}\ph_{abcj} \theta_i\theta_j=0$ which substituted into \eqref{rdt} gives
\begin{equation}\label{rdt2}
Ric_{ij}\theta_i\theta_j=-\frac32\theta_i\sb_i||\theta||^2.
\end{equation}
Insert \eqref{rdt2} into \eqref{gafinf} to get 
\begin{equation}\label{gafinf1}
\Delta||\theta||^2+3\theta_i\sb_i||\theta||^2=-||\sb\theta||^2\le 0.
\end{equation}
We  apply  the strong maximum principle to \eqref{gafinf1} (see e.g. \cite{YB,GFS}) to achieve $d||\theta||^2=\sb\theta=0$.
\end{proof}
The proof of Theorem~\ref{mainc} follows from  Theorem~\ref{main} and Theorem~\ref{gaud}.

%vanishes because the Lee form is zero, $0=\theta_pF_{ps}=T_{a\bar a s}$. Hence, $||\sb C||^2=0$ because of \eqref{inst17}.

\subsection{Proof of Theorem~\ref{hul}}

We recall that the Hull connection $\sb^h$ is defined to be the metric connection with torsion $-T$, where $T$ is the torsion of the torsion connection $\sb$,
\begin{equation}\label{hu}\sb^h=\LC-\frac12T=\sb-T.
\end{equation}
%Concerning the $G_2$-Hull connection, we prove the following
%\begin{thrm}\label{hul}
%The curvature $R^h$ of the $G_2$-Hull connection $\sb^h$ is a $G_2$ instanton if and only if the torsion is closed, $dT=0$.
%\end{thrm}
\begin{proof}
We start with the general well-known formula for the curvatures of two metric connections with totally skew-symmetric torsion $T$ and $-T$, respectively, see e.g. \cite{MS}, which applied to the curvatures of the torsion connection and the $SU(3)$-Hull connection reads
\begin{equation}\label{hust}
R(X,Y,Z,V)-R^h(Z,V,X,Y)=\frac12dT(X,Y,Z,V).
\end{equation}
If $dT=0$ the result was already  observed  in \cite{MS}. Indeed, in this case  the $SU(3)$-Hull connection is an $SU(3)$--instanton since $\sb$ preserves the $SU(3)$ structure and the holonomy group of $\sb$ is contained in the Lie algebra $\frak{su}(3)$ \cite{MS}.

For the converse, \eqref{hust}, the  $SU(3)$--instanton condition \eqref{rr1} applied for  $R^h$ and \eqref{rr} yield
\begin{equation}\label{huin1}
\begin{split}
dT_{iabc}\ps_{abc}=R_{iabc}\ps_{abc}+R^h_{bcai}\ps_{abc}=0 \,,\\
dT_{iabc}\sp_{abc}=R_{iabc}\sp_{abc}+R^h_{bcai}\sp_{abc}=0 \,,\\
dT_{iabc}\Phi_{jabc}=R_{iabc}\Phi_{jabc}+R^h_{bcai}\Phi_{jabc}=-2R_{iaja}-2R^h_{jaai}=2Ric_{ij}-2Ric^h_{ji}=0 \,,
\end{split}
\end{equation}
where $Ric^h$ is the Ricci tensor of the $SU(3)$-Hull connection and  the trace of \eqref{hust} gives $Ric(X,V)-Ric^h(V,X)=0.$
Now, \eqref{huin1} show that for any vector field $X$ the 3-form $(X\lrcorner dT)\in \Lambda^3_{12}$. Then the 4-form $dT=0$ due to Proposition~\ref{4-for}.
%together with Proposition~\ref{4-for} imply $dT=0$.
\end{proof}

\section{Examples} 
\subsection{Complex case}\label{Ex-Complex-case}
%{\color{blue} 
It is clear due to \cite[Lemma~3.4]{I} that a CYT space  with parallel torsion with respect to the Strominger-Bismut connection  the curvature of the Strominger-Bismut connection is an  $SU(3)$--instanton. \cite[Theorem~1.1]{ZZ4} describes  all compact balanced complex 6 manifold with parallel torsion with respect to the Strominger-Bismut connection. 
Some of these spaces have holomorphically trivial canonical bundle, therefore are CYT (see \cite{OU,ZZ4}), and  the curvature of the Strominger-Bismut connection is an  $SU(3)$--instanton.  Corollary~\ref{comm} combined with  \cite[Theorem~1.1]{ZZ4} and \cite[Theorem~1.16]{ZZ3} describes all possible balanced CYT spaces with  $SU(3)$--instanton Strominger-Bismut connection. 

%The non-balanced case with $\sb T=0$ is treated in \cite[Theorem~1.16]{ZZ3}. For example, Vaisman 3-folds with  abelian holonomy of the Strominger-Bismut connection are CYT \cite{ZZ3,ZZ4} and therefore the Strominger-Bismut connection is an instanton.

\vskip.15cm

Next we present those left-invariant balanced CYT solutions with holomorphically trivial canonical bundle and parallel torsion constructed from unimodular Lie groups, with special attention to their role in the Hull-Strominger system and the heterotic equations of motion. We will essentially follow the lines of \cite{FIUV} and \cite{OUV}.

\vskip.2cm

\noindent$\bullet$ {\bf The nilmanifold ${\mathfrak h}_3$}.  
For every $t\in \mathbb{R}-\{0\}$, let us consider the nilpotent Lie algebra ${\mathfrak n}_t$ defined by 
a basis of 1-forms $\{e_1,\dots, e_6\}$ satisfying the equations
\begin{gather}\label{ecus-h3}
de_1=de_2=de_3=de_4=de_5=0,\qquad de_6= - 2 t ( e_1\wedge e_2 - e_3\wedge e_4).
\end{gather}

We consider the metric given by $g=\sum_{i=1}^6e_i^2$, and the $SU(3)$-structure 
defined as 
$$
\omega=e_{12}+e_{34}+e_{56},
\qquad \Theta=(e_{1}+\sqrt{-1}\, e_{2})\wedge (e_{3}+\sqrt{-1}\, e_{4})\wedge (e_{5}+\sqrt{-1}\, e_{6}).
$$
that is, we are considering the $SU(3)$-structure 
$(\omega=-F,\Theta=-\Psi)$ with respect to the one given in~\eqref{AA}. 

It is easy to see that ${\mathfrak n}_t\cong {\mathfrak n}_{t'}$ (i.e. they are isomorphic as Lie algebras), for every $t,t'\in \mathbb{R}-\{0\}$. In particular, all the Lie algebras are isomorphic 
to ${\mathfrak n}_{-\frac12}$, which is precisely the nilpotent Lie algebra labelled as ${\mathfrak h}_3=(0,0,0,0,0,12-34)$ in \cite{FIUV,OUV}. 
It is well known that the connected simply connected and nilpotent
Lie group $H_3$ corresponding to ${\mathfrak h}_3$ has a lattice $\Gamma$ of maximal rank.

On the other hand, we note that the $SU(3)$-structures defined on different ${\mathfrak n}_t$'s are not isomorphic. 
So, the $SU(3)$-structure given above on the family ${\mathfrak n}_t$ can alternatively be thought as 
a 1-parameter family of $SU(3)$-structures on the $6$-dimensional nilmanifold $M=\Gamma\backslash H_3$, induced by the 1-parameter family of left-invariant $SU(3)$-structures on the Lie group $H_3$. 

\vskip.1cm

It is easy to verify using  \eqref{cy1} and \eqref{ecus-h3} that the almost complex structure $J$ is integrable, $N=0$. Furthermore, $d\Theta=0$ and 
$d\omega=2t(e^{125}-e^{345})$, which implies $d \omega^2= 0$, so the canonical bundle is holomorphically trivial and the Hermitian metrics are balanced, $\theta=0$.

The torsion of the Strominger-Bismut connection is $T=-2t(e^{126}-e^{346})$.
%%%$dT=-8t^2 e^{1234}$.
A direct calculation shows that the nonzero  terms
of the Strominger-Bismut connection are 
\begin{gather}\label{nablas-h3}
\nabla_{e_6}e_1= -2 t \,e_2, \qquad  \nabla_{e_6}e_2= 2 t \,e_1, \qquad \nabla_{e_6}e_3= 2 t \,e_4, \qquad \nabla_{e_6}e_4= -2 t \,e_3. \nonumber
\end{gather}
Now, it is easy to verify that $\nabla\Theta =0$, 
moreover, by \cite{OUV} the holonomy of the Strominger-Bismut connection is $U(1)\subset SU(3)$. 
It is also easy to check that the torsion tensor $T$ is $\sb$-parallel and therefore the curvature of $\sb$ satisfies \eqref{r4}, i.e. the Strominger-Bismut connection is an  $SU(3)$--instanton. This was already noticed in \cite[Section 5]{FIUV}, where this (non-flat)  $SU(3)$--instanton is used to construct compact solutions to the Hull-Strominger system which, by \cite{I1}, are also solutions to the heterotic equations of motion.
%%%(Compare with \cite[Proposition 5.9]{ZZ4}.)

In conclusion, the nilmanifold $M=\Gamma\backslash H_3$, constructed from ${\mathfrak h}_3$, is a compact complex three-fold with holomorphically trivial canonical bundle endowed with a one-parameter family of balanced CYT $SU(3)$-structures such that the Strominger-Bismut connection is an  $SU(3)$--instanton, hence with parallel torsion.
All these solutions are invariant on the nilmanifold, i.e. they come from left-invariant ones on the Lie group $H_3$, and they all provide solutions to the heterotic equations of motion.

\vskip.2cm

\noindent$\bullet$ {\bf The solvmanifold ${\mathfrak g}_7$}. 
For $\delta=\pm 1$ and for every $r,t\in \mathbb{R}-\{0\}$, let us consider the solvable Lie algebra ${\mathfrak s}^{\delta}_{r,t}$ defined by 
a basis of 1-forms $\{e_1,\dots, e_6\}$ satisfying the equations
\begin{gather}\label{ecus-g7}
de_1= -\frac{2}{t}\, e_{2}\wedge e_{5},\quad 
de_2= \frac{2}{t}\, e_{1}\wedge e_{5},\quad 
de_3=  \frac{2}{t}\, e_{4}\wedge e_{5},\\\nonumber
de_4= - \frac{2}{t}\, e_{3}\wedge e_{5},\quad
de_5= 0, \quad
de_6 = -\frac{2 \delta t}{r^2}\,\left( e_{1}\wedge e_{2} - e_{3}\wedge e_{4} \right).
\end{gather}

As for the nilmanifold case, let us consider the metric given by $g=\sum_{i=1}^6e_i^2$, and the $SU(3)$-structure 
defined as 
$$
\omega=e_{12}+e_{34}+e_{56},
\qquad \Theta=(e_{1}+\sqrt{-1}\, e_{2})\wedge (e_{3}+\sqrt{-1}\, e_{4})\wedge (e_{5}+\sqrt{-1}\, e_{6}).
$$ 

It can be checked that all the Lie algebras ${\mathfrak s}^{\delta}_{r,t}$ are isomorphic 
to the solvable Lie algebra labelled as ${\mathfrak g}_7=(24+35,46,56,-26,-36,0)$ in \cite{OUV}. 
We note that the Lie algebra ${\mathfrak g}_7$ admits precisely two complex structures $J_{\pm}$ with non-zero closed $(3,0)$-form, and this fact in encoded in the value of $\delta\in\{-1,1\}$. 
The connected simply-connected solvable Lie group $G_7$ corresponding to ${\mathfrak g}_7$ admits a lattice $\Gamma$ of maximal rank (see for instance \cite{OUV} and references therein).
Since the $SU(3)$-structures defined on different ${\mathfrak s}^{\delta}_{r,t}$ are not isomorphic, similarly to the nilmanifold case above, we have a family of $SU(3)$-structures on the $6$-dimensional solvmanifold $M=\Gamma\backslash G_7$, induced by the family of left-invariant $SU(3)$-structures on the Lie group $G_7$. 

\vskip.1cm

It is easy to verify using  \eqref{cy1} and \eqref{ecus-g7} that the almost complex structures $J_{\pm}$ are integrable. Moreover, $d\Theta=0$ and 
$d\omega=\frac{2 \delta\, t}{r^2} \, (e^{125}-e^{345})$, which implies $d \omega^2= 0$, so the canonical bundle is holomorphically trivial and the Hermitian metrics are balanced.

The torsion 3-form is given by $T=-\frac{2 \delta\, t}{r^2} \, (e^{126}-e^{346})$.
%%%$dT= -\frac{8\, t^2}{r^4} e^{1234}$
A direct calculation shows that the nonzero  terms
of the Strominger-Bismut connection are 
\begin{gather}\label{nablas-g7}\nonumber
\nabla_{e_5}e_1= \frac{2}{t} \,e_2, \qquad  \nabla_{e_5}e_2= -\frac{2}{t}  \,e_1, \qquad \nabla_{e_5}e_3= -\frac{2}{t}  \,e_4, \qquad \nabla_{e_5}e_4= \frac{2}{t}  \,e_3, \\\nonumber
\nabla_{e_6}e_1= -\frac{2\delta t}{r^2}  \,e_2, \qquad  \nabla_{e_6}e_2= \frac{2\delta t}{r^2} \,e_1, \qquad \nabla_{e_6}e_3= \frac{2\delta t}{r^2} \,e_4, \qquad \nabla_{e_6}e_4= -\frac{2\delta t}{r^2} \,e_3.
\end{gather}

One has $\nabla\Theta =0$, 
furthermore, by \cite[Proposition 6.1]{OUV} the holonomy of the Strominger-Bismut connection reduces to $U(1) \subset SU(3)$.
We also get that the torsion tensor $T$ is $\sb$-parallel and therefore the curvature of $\sb$ satisfies \eqref{r4}, i.e. the Strominger-Bismut connection is an  $SU(3)$--instanton. 
This  $SU(3)$--instanton property was already obtained in \cite[Proposition 5.2]{OUV}, where it is also found that the  $SU(3)$--instanton is non-flat. These properties are used in \cite[Theorem 5.3]{OUV}  to construct compact solutions to the Hull-Strominger system and the heterotic equations of motion on the solvmanifold $M=\Gamma\backslash G_7$.
%%%(Compare with \cite[Proposition 5.9]{ZZ4}.)

In conclusion, the solvable Lie algebra ${\mathfrak g}_7$ gives rise to compact complex three-folds with holomorphically trivial canonical bundle endowed with a family of balanced CYT $SU(3)$-structures for which the Strominger-Bismut connection is an  $SU(3)$--instanton, hence with parallel torsion.
Again, all these solutions are invariant on the solvmanifold.

\vskip.2cm

\noindent$\bullet$ {\bf The quotient of the simple complex Lie group ${\rm SL(2},\bC)$}. 
For every $t\in \mathbb{R}-\{0\}$, let us consider the Lie algebra ${\mathfrak g}_t$ defined by 
a basis of 1-forms $\{e_1,\dots, e_6\}$ satisfying the equations
\begin{gather}\label{ecus-sl2C}
de_1 =\frac{1}{t} \left(e_{3}\wedge e_{5} - e_{4}\wedge e_{6}\right),\quad 
de_2=\frac{1}{t} \left(e_{3}\wedge e_{6} + e_{4}\wedge e_{5}\right),\quad 
de_3= -\frac{1}{t} \left(e_{1}\wedge e_{5} - e_{2}\wedge e_{6}\right),\\\nonumber
de_4= -\frac{1}{t} \left(e_{1}\wedge e_{6} + e_{2}\wedge e_{5}\right),\quad
de_5= \frac{1}{t} \left(e_{1}\wedge e_{3} - e_{2}\wedge e_{4}\right), \quad
de_6 = \frac{1}{t} \left(e_{1}\wedge e_{4} + e_{2}\wedge e_{3}\right).
\end{gather}

As in the previous cases, we consider the metric given by $g=\sum_{i=1}^6e_i^2$ and the $SU(3)$-structure 
$$
\omega=e_{12}+e_{34}+e_{56},
\qquad \Theta=(e_{1}+\sqrt{-1}\, e_{2})\wedge (e_{3}+\sqrt{-1}\, e_{4})\wedge (e_{5}+\sqrt{-1}\, e_{6}).
$$ 

It easy to see that all the Lie algebras ${\mathfrak g}_t$ are isomorphic 
to $\mathfrak{so}(3,\!1)$, the real $6$-dimensional Lie algebra underlying the simple complex Lie algebra $\mathfrak{sl}(2,\bC)$. The complex structure on $\mathfrak{so}(3,\!1)$ is the canonical one coming from $\mathfrak{sl}(2,\bC)$. 
It is known that the simple complex Lie group ${\rm SL(2},\bC)$ admits a lattice $\Gamma$ of maximal rank, so we have a one-parameter family of $SU(3)$-structures on the quotient manifold  $M=\Gamma\backslash {\rm SL(2},\bC)$, induced by the 1-parameter family of left-invariant $SU(3)$-structures on the Lie group ${\rm SL(2},\bC)$. Notice that the 1-parameter family of metrics is a scaling of  the standard metric induced by the Killing form (which corresponds to $t=1$). 

\vskip.1cm

Using \eqref{ecus-sl2C} it is easy to check that $d\Theta=0$ and 
$d\omega= -\frac{1}{t}(e^{136} + e^{145} + e^{235} + 3\, e^{246})$, which implies $d \omega^2= 0$. 
So, the canonical bundle is holomorphically trivial and the Hermitian metrics are balanced.

The torsion 3-form is given by $T= -\frac{1}{t}(3\, e^{135} + e^{146} + e^{236} + e^{245})$.
%%%$dT= -\frac{4}{t^2} (e^{1234} + e^{1256} + e^{3456})$
Now, a direct calculation shows that the nonzero  terms
of the Strominger-Bismut connection are 
\begin{gather}\label{nablas-sl2C}\nonumber 
\nabla_{e_1}e_3= -\frac{2}{t} \,e_5, \qquad  \nabla_{e_1}e_4= -\frac{2}{t}  \,e_6, \qquad \nabla_{e_1}e_5= \frac{2}{t}  \,e_3, \qquad \nabla_{e_1}e_6= \frac{2}{t}  \,e_4, \\\nonumber
\nabla_{e_3}e_1= \frac{2}{t} \,e_5, \qquad  \nabla_{e_3}e_2= \frac{2}{t} \,e_6, \qquad \nabla_{e_3}e_5= -\frac{2}{t} \,e_1, \qquad \nabla_{e_3}e_6= -\frac{2}{t} \,e_2,\\\nonumber 
\nabla_{e_5}e_1= -\frac{2}{t} \,e_3, \qquad  \nabla_{e_5}e_2= -\frac{2}{t}  \,e_4, \qquad \nabla_{e_5}e_3= \frac{2}{t} \,e_1, \qquad \nabla_{e_5}e_4= \frac{2}{t}  \,e_2.
\end{gather}

It is easy to verify that  $\nabla\Theta =0$, 
moreover, by \cite[Proposition 6.3]{OUV} the holonomy of the Strominger-Bismut connection reduces to $SO(3)$ inside $SU(3)$. 
One also gets that the torsion tensor $T$ is $\sb$-parallel and therefore the curvature of $\sb$ satisfies \eqref{r4}, i.e. the Strominger-Bismut connection is an  $SU(3)$--instanton. 
This  $SU(3)$--instanton property was already found in \cite[Proposition 4.1]{OUV}, where it is also proved that the  $SU(3)$--instanton is not flat. These properties are
used in \cite[Theorem 4.3]{OUV} to construct compact solutions to the Hull-Strominger system and the heterotic equations of motion on the quotient manifold  $M=\Gamma\backslash {\rm SL(2},\bC)$.
%%%(Compare with \cite[Proposition 5.9]{ZZ4}.) 
%%%By \cite[Proposition 2.6]{ZZ4}, for a quotient of SL(2,$\bC$) the only possible metrics with parallel $T$ are those equipped with a scaling of $g_0$ (in the notation of \cite{ZZ4}),  where $g_0$ is the standard metric induced by the Killing form. 

Therefore, the simple complex Lie algebra $\mathfrak{sl}(2,\bC)$ gives rise to compact complex three-folds with holomorphically trivial canonical bundle endowed with a family of balanced CYT $SU(3)$-structures for which the Strominger-Bismut connection is an  $SU(3)$--instanton, hence with parallel torsion.
As in the previous cases, all these solutions are invariant on the quotient manifold.

\begin{rmrk}\label{ou-paper}
Let $M=\Gamma\backslash G$ be a six-dimensional compact manifold defined as the quotient of a connected simply connected Lie group $G$ by a lattice $\Gamma$ of maximal rank.
Suppose that $M$ possesses an invariant balanced Hermitian structure $(J,F)$ with invariant non-zero closed $(3,0)$-form $\Psi$.  
Let $\nabla^{\varepsilon,\rho}_{(J,F)}$ be any metric connection in the $(\varepsilon,\rho)$-plane generated by the Levi-Civita connection and the Gauduchon line of Hermitian connections (i.e. the line passing through the Chern connection and the Strominger-Bismut connection). This plane was introduced and studied in \cite{OUV}, and it also contains the Hull connection. 

By \cite[Theorem 3.5]{OU}, if $\nabla^{\varepsilon,\rho}_{(J,F)}$ is a non-flat  $SU(3)$--instanton, then the Lie algebra of $G$ is isomorphic to ${\mathfrak h}_3$, ${\mathfrak g}_7$, or $\mathfrak{so}(3,\!1)$, i.e. the real Lie algebra underlying $\mathfrak{sl}(2,\bC)$. 

In particular, by Corollary~\ref{comm} and the above result for $(\varepsilon,\rho)=(\frac12,0)$ (which corresponds precisely to the Strominger-Bismut connection), we arrive at the fact that these are the only unimodular Lie algebras admitting complex structures with non-zero closed $(3,0)$-form and balanced Hermitian metrics such that the torsion form is parallel. We note that this agrees with the results obtained by Zhao and Zheng in \cite{ZZ4} for balanced BPT (Bismut torsion parallel) Hermitian Lie algebras.
 \end{rmrk}

%}

\subsection{Non-complex case} 
All (strict) Nearly K\"ahler 6-manifold are examples of ACYT 6 manifolds with an  $SU(3)$--instanton torsion connection since the torsion and the Nijenhuis tensor are $\sb$--parallel due to the result of Kirichenko \cite{Kir} (see also \cite{BM}).

The next example, taken  from \cite{II}, is an example of a balanced ACYT 6-manifold with  $SU(3)$--instanton torsion connection which is neither complex nor a Nearly K\"ahler 6-manifold.

Let $G$ be the six-dimensional connected simply connected and nilpotent
 Lie group, determined by the left-invariant 1-forms $\{e_1,\dots,e_6\}$
 such that
\begin{gather}\label{in1}\nonumber
de_2=de_3=de_6=0,\\\nonumber
de_1=e_3\wedge e_6,\quad
de_4= e_2\wedge e_6, \quad
de_5= e_2\wedge e_3.
\end{gather}
Consider the metric on $G\cong \mathbb R^6$ defined by
$g=\sum_{i=1}^6e_i^2$.
Let $(F,\Psi)$ be the $SU(3)$-structure on $G$ given by \eqref{AA}. Then$(G,F,\Psi)$ is
an almost complex manifold with a $SU(3)$-structure.

It is easy to verify using  \eqref{cy1} and \eqref{in1} that
\begin{gather}\label{in2}
dF=-3e_{236}, \quad  N=-\Psi^-, \quad d\Psi^-=*F, \quad
(N,\Psi^-)=-4, \\ \nonumber \theta=d\Psi^+=(N,\Psi^+)=0.
\end{gather}

Hence, $(G,\Psi,g,J)$ is neither complex nor Nearly K\"ahler manifold but it fulfills
the conditions \eqref{cycon} of Theorem~\ref{cythm1} and
therefore it is an ACYT 6 manifold, i.e. there exists a $SU(3)$ connection with torsion 3-form on $(G,\Psi,g,J)$.
The expression
\eqref{torcy} and \eqref{in2} give
\begin{equation}\label{tor}
T=-2e_{145}+e_{136}+e_{235}-e_{246}, \quad dT=-2(e_{1256}+e_{3456}+e_{1234})=2*F=-2\Phi.
\end{equation}
For the Ricci tensor $Ric$ 
 of the torsion connection  we calculate using \eqref{ricdt}, %\eqref{ssc} 
 and the second identity in \eqref{tor} applying \eqref{iden}
\begin{equation}\label{ain}
\begin{split}
Ric_{ij}=\frac1{12}dT_{abci}\Phi_{abcj}=-\frac16\Phi_{abci}\Phi_{abcj}=-2g_{ij} \,.
%\\\rho_{ij}=Ric_{is}F_{sj}-\frac14dT_{ijkl}F_{kl}=-2F_{ij}+\frac12\Phi_{ijkl}F_{kl}=0.
\end{split}
\end{equation}
Hence, the ACYT torsion connection is Einstein with  constant negative  scalar curvature $s=-12$ and positive Riemannian scalar curvature $s^g=s +\frac14||T||^2=1$ which agrees with \cite[Theorem~4.2]{IS}.

Note that, in general, if a metric connection with torsion 3-form is Einstein its scalar curvature is not necessarily constant \cite{AFer}.

The equalities  \eqref{tor} and  \eqref{cy2} imply  the nonzero  terms
of the torsion connection are (c.f. \cite{II})
\begin{gather}\label{tor1}
\nabla_{e_1}e_6=-e_3, \qquad  \nabla_{e_5}e_2= e_3, \qquad \nabla_{e_4}e_6=-e_2, \\\nonumber
\nabla_{e_4}e_5=-e_1, \qquad \nabla_{e_5}e_1=-e_4, \qquad \nabla_{e_1}e_4=-e_5.
\end{gather}
It is easy to verify (cf. \cite{II}) that 
the torsion tensor $T$ as well as the Nijenhuis tensor $N$ and $dF$ are $\sb$-parallel and therefore the curvature of the torsion connection satisfies \eqref{r4} and  it is an  $SU(3)$--instanton.

The coefficients of the structure equations of the Lie algebra given by \eqref{in1}
are integers. Therefore, the well-known theorem of Malcev \cite{Mal} states that
the group $G$ has a uniform discrete subgroup $\Gamma$ such that $Nil^6=\Gamma\backslash G$ is a compact 6-dimensional nil-manifold. The $SU(3)$-structure, described above, descends to $Nil^6$ and therefore we obtain a compact example of a balanced non complex ACYT $\sb$--Einstein  space with  $SU(3)$--instanton curvature.

\section{Appendix}\label{cins}
\subsection{Chern instanton}
It is known that on a compact  non-K\"ahler CYT manifold with an exact  Lee form $\theta=df$ for a smooth function $f$, if  the curvature of the  Chern connection is  an $SU(n)$-instanton then $f$ must be  a constant.

We include here a proof of this fact for completeness  based on \cite{IP2}. We have
\begin{prop}\label{chern}
On a 2n-dimensional compact CYT manifold if the curvature of the Chern connection is an $SU(n)$-instanton then it is balanced.
\end{prop}
\begin{proof}
Let on a compact CYT space $K$ be the curvature of the Chern connection and $C=\frac12(T(JX,JY,Z)-T(X,Y,Z))$ be its torsion. Then \cite[Proposition~3.3]{IP2} gives the formula 
\[\frac12K(e_i,Je_i,JX,Y)=\rho^{1,1}(JX,Y)+C(X, e_i,e_j)C(Y,e_i,e_j) -\frac14dT(JX,Y,e_i,Je_i) ,\]
where $\rho^{1,1}$ is the (1,1) part of $\rho$. Note that $\rho=0$ on a CYT manifold.

If $K$ is an $SU(n)$-instanton then we get taking the trace in the above formula 
\[ 0=||C||^2 
-\frac14dT(Je_j,e_j,e_i,Je_i)
= \frac13||T||^2 
 -\frac13||T||^2+2||\theta||^2+2\delta\theta 
= 2||\theta||^2+2\delta\theta,%=\frac23||T||^2-2||df||^2-\delta df,
\]
where we used the identities $||C||^2=\frac13||T||^2$ and 
 $dT(e_j,Je_j,e_i,Je_i)=8||\theta||^2+8\delta\theta-\frac43||T||^2$
(see \cite[(3.24)]{IP2}.  An integration over the compact space yields $\theta=0$.
\end{proof}

%%%%%%%%%%%%%%%%%%%%

\end{document}